\documentclass[letterpaper,12pt]{amsart}
\usepackage{amscd, amssymb}
\usepackage[driverfallback=hypertex]{hyperref}
\usepackage{enumitem}
\setlist[enumerate,1]{label={(\alph*)}}
\input xy
\xyoption{all}

\newtheorem{tm}{Theorem}[section]
\newtheorem{pr}[tm]{Proposition}
\newtheorem{lm}[tm]{Lemma}
\newtheorem{cy}[tm]{Corollary}

\theoremstyle{definition}
\newtheorem{df}[tm]{Definition}

\theoremstyle{remark}
\newtheorem{rem}[tm]{Remark}
\newtheorem{ex}[tm]{Example}

\newcommand{\R}{\mathbb R}

\renewcommand{\L}{\mathcal L}

\newcommand{\MM}{\mathcal M}
\newcommand{\M}{\mathcal M}

\renewcommand{\H}{\mathcal H}

\DeclareMathOperator{\im}{Im}
\DeclareMathOperator{\re}{Re}
\DeclareMathOperator{\id}{Id}
\DeclareMathOperator{\vol}{vol}

\DeclareMathOperator{\Vol}{Vol}

\DeclareMathOperator{\supp}{supp}

\DeclareMathOperator{\ham}{Ham}
\DeclareMathOperator{\cal}{Cal}
\DeclareMathOperator{\flux}{Flux}
\DeclareMathOperator{\tr}{tr}
\DeclareMathOperator{\herm}{Herm}
\DeclareMathOperator{\symp}{Symp}
\newcommand{\CC}{\mathcal C}
\newcommand{\OO}{\mathcal O}
\newcommand{\HH}{\mathcal H}

\newcommand{\X}{\mathcal X}
\newcommand{\G}{\mathcal G}
\newcommand{\I}{\mathcal I}
\newcommand{\E}{\mathcal E}

\begin{document}
\title{The Calabi homomorphism, Lagrangian paths and special Lagrangians}
\author{Jake P. Solomon}
\date{April 2013}
\begin{abstract}
Let $\OO$ be an orbit of the group of Hamiltonian symplectomorphisms acting on the space of Lagrangian submanifolds of a symplectic manifold $(X,\omega).$ We define a functional $\CC:\OO \to \R$ for each differential form $\beta$ of middle degree satisfying $\beta \wedge \omega = 0$ and an exactness condition. If the exactness condition does not hold, $\CC$ is defined on the universal cover of~$\OO.$ A particular instance of $\CC$ recovers the Calabi homomorphism. If $\beta$ is the imaginary part of a holomorphic volume form, the critical points of $\CC$ are special Lagrangian submanifolds. We present evidence that $\CC$ is related by mirror symmetry to a functional introduced by Donaldson to study Einstein-Hermitian metrics on holomorphic vector bundles. In particular, we show that $\CC$ is convex on an open subspace $\OO^+ \subset \OO.$ As a prerequisite, we define a Riemannian metric on $\OO^+$ and analyze its geodesics. Finally, we discuss a generalization of the flux homomorphism to the space of Lagrangian submanifolds, and a Lagrangian analog of the flux conjecture.
\end{abstract}

\maketitle

\pagestyle{plain}

\tableofcontents

\section{Introduction}
\subsection{The Calabi homomorphism}

The starting point of our paper is the Calabi homomorphism. Let $(M,\omega_M)$ be a symplectic manifold of dimension $2m.$ Let $\ham(M,\omega_M)$ denote the group of compactly supported Hamiltonian symplectomorphisms of $(M,\omega_M).$
The universal cover $\widetilde\ham(M,\omega_M)$ is the space of smooth paths $\phi = \{\phi_t\}_{t\in [0,1]}$ in $\ham(M,\omega)$ with $\phi_0 = \id_M,$ modulo end-point preserving homotopy.
Assume for the moment that $M$ is non-compact. Let 
\[
[\phi] \in \widetilde \ham(M,\omega_M)
\]
be the equivalence class of a path $\phi$. Let
\[
H_t : M \to \R, \qquad t \in [0,1],
\]
be the time-dependent Hamiltonian generating $\phi$ such that $\supp(H_t)$ is compact for all $t.$
Calabi~\cite{Ca70} observed that the functional
\[
\cal : \widetilde\ham(M,\omega_M) \to \R
\]
given by
\[
\cal([\phi]) = \int_0^1\int_M H_t \omega^n dt
\]
is well defined. That is, it depends only on the homotopy class of $\phi.$ In fact, $\cal$ is a homomorphism. See~\cite{Ba78,MS98} for further discussion. It is essential in the preceding discussion that $M$ be non-compact. Indeed, for $M$ compact, any homomorphism $\widetilde\ham(M,\omega_M) \to \R$ must be trivial by a theorem of Banyaga~\cite{Ba78}.

Let $X = M \times M$ and $\omega = -\omega_M \boxplus \omega_M.$ The graph of a symplectomorphism of $M$ is a Lagrangian submanifold of $X.$ So, a path $\phi = \{\phi_t\}_{t \in [0,1]}$ in $\ham(M,\omega),$ corresponds to a path of Lagrangian submanifolds $\Lambda = \{\Lambda_t\}_{t \in [0,1]},$ where
\[
\Lambda_t  = \{(x,\phi_t(x))|x \in M\} \subset X.
\]
This correspondence suggests that $\cal$ arises from a more general homotopy invariant of Lagrangian paths. The present paper constructs such an invariant. It turns out that the generalization of $\cal$ to Lagrangian paths is non-trivial for both non-compact and compact manifolds.

Let $(X,\omega)$ be a general symplectic manifold of dimension $2n.$ Let $L$ be an oriented manifold of dimension $n$ and let $d \in H_n(X).$
Let $\L = \L(X,L,d)$ denote the space of oriented Lagrangian submanifolds $\Lambda \subset X$ diffeomorphic to $L,$ and in the case $L$ is compact, having fundamental class $d.$ When $L$ is non-compact, we impose that all $\Lambda \in \L$ coincide with a given $\Lambda_0$ outside a compact subset. Following Akveld-Salamon~\cite{AS01}, we treat $\L$ as an infinite dimensional manifold. The tangent space to $\L$ at $\Lambda \in \L$ is canonically identified with the space of closed $1$-forms on $\Lambda.$ Thus, for a given path $\Lambda = \{\Lambda_t\}_{t \in [0,1]}$ in $\L,$ the time derivative $\frac{d}{dt} \Lambda_t$ is a closed $1$-form on $\Lambda_t.$ The path $\Lambda$ is called \emph{exact} if for each $t$ there exists $h_t : \Lambda_t \to \R$ such that $\frac{d}{dt} \Lambda_t = dh_t.$ In this case, we call $h_t$ an associated family of Hamiltonian functions. An exact path is called \emph{compactly supported} if there exists a compactly supported associated family of Hamiltonian functions.

Let $\beta$ be a closed $n$-form such that $\omega \wedge \beta = 0.$ Let $\Lambda = \{\Lambda_t\}_{t \in [0,1]}$ be an exact compactly supported Lagrangian path in $\L(X,L,d).$ Let $h = \{h_t\}_{t \in [0,1]}$ be an associated family of Hamiltonian functions. If $L$ is non-compact, assume that $h_t$ is normalized so that $\supp(h_t)$ is compact. If $L$ is compact, assume that $\int_d \beta = 0.$ We define
\[
\CC(\Lambda) = \int_0^1 dt \int_{\Lambda_t} h_t \beta.
\]
Our main theorem is the following.
\begin{tm}\label{tm:main} The functional $\CC(\Lambda)$
depends only on the endpoint preserving homotopy class of $\Lambda.$
\end{tm}
Theorem~\ref{tm:main} implies that $\CC$ descends to a functional on the space of end-point preserving homotopy classes of exact compactly supported Lagrangian paths. We write $\CC([\Lambda]) = \CC(\Lambda)$. In particular, let $\OO$ be an orbit of $\ham(X,\omega)$ acting on $\L(X,L,d).$ Choosing a base-point $\Lambda_* \in \OO,$ identify the universal cover $\widetilde \OO$ with the space of paths in $\OO$ starting at $\Lambda_*$, modulo end-point preserving homotopy. Any path in $\OO$ is exact and compactly supported. Thus we have defined a functional
\[
\CC : \widetilde \OO \to \R.
\]

As the following example indicates, it is not hard to find forms $\beta$ as required for the definition of $\CC.$
\begin{ex}\label{ex:prim}
Let $(X,\omega)$ be a compact K\"ahler manifold. Let $\beta$ be the harmonic representative of a primitive $n^{th}$ cohomology class. Recall that the Laplace-Beltrami operator commutes with exterior multiplication by $\omega.$ So, the Lefschetz decomposition implies that $\omega \wedge \beta = 0.$ Similarly, for general symplectic manifolds, we can take $\beta$ a $d + d^\Lambda$-primitive harmonic $n$-form in the sense of~\cite{TY09}.
\end{ex}
\begin{tm}\label{tm:exact}
Suppose there exist forms $\lambda,\gamma$ and a constant $c \neq -1$ such that
\begin{gather}
d\lambda = \omega, \qquad d\gamma = \beta, \notag\\
\lambda \wedge \beta = -c \omega \wedge \gamma.\label{eq:lbog}
\end{gather}
Then
\[
\CC(\Lambda) = \frac{1}{c+1} \left(\int_{\Lambda_0} \lambda\wedge \gamma - \int_{\Lambda_1} \lambda \wedge \gamma\right).
\]
So, $\CC(\Lambda)$ depends only on $\Lambda_0$ and $\Lambda_1.$
\end{tm}
In particular, Theorem~\ref{tm:exact} implies that $\CC$ further descends to a functional
\[
\CC : \OO \to \R.
\]
The following example motivates condition~\eqref{eq:lbog}.
\begin{ex}\label{ex:exact}
Suppose $\omega = d\lambda.$ Let $\xi$ be the Liouville vector-field corresponding to $\lambda.$ That is, $i_\xi \omega = \lambda.$ Suppose that
\[
L_\xi \beta = c\beta.
\]
with $c \neq 0.$ Take $\gamma = i_\xi \beta/c.$ Then
\[
0 = i_\xi(\omega \wedge \beta) = \lambda \wedge \beta + c \omega \wedge \gamma.
\]
implying condition~\eqref{eq:lbog}.
\end{ex}

In Section~\ref{sec:class}, we give explicit $\beta,\gamma,$ for which Theorems~\ref{tm:main} and~\ref{tm:exact} recover known properties of the Calabi homomorphism. In particular, $\CC$ is shown to generalize the Calabi homomorphism.

Another generalization of the Calabi homomorphism in the context of compact symplectic manifolds is the notion of a Calabi quasimorphism due to Entov-Polterovich \cite{EP03}. Though their construction is fundamentally different from that of the present paper, it is tempting to speculate on a possible connection.

\subsection{Special Lagrangians and mirror symmetry}

In Section~\ref{sec:slag}, we take $X$ a Calabi-Yau manifold and $\beta$ the imaginary part of a holomorphic volume form $\Omega.$ Recall that Harvey and Lawson~\cite{HL82} called a Lagrangian submanifold $\Gamma \subset X$ \emph{special} if $\im \Omega|_\Gamma = 0.$ They showed that special Lagrangians are volume minimizing. Particularly following the conjecture of Strominger-Yau-Zaslow on mirror symmetry~\cite{SY96}, special Lagrangians have received considerable attention~\cite{Hi97,Jo02,Jo03,TY02}. We show the functional $\CC$ gives a variational principle for special Lagrangians. Namely, varying one endpoint of the path $\Lambda$ while holding the other fixed, the critical points of $\CC([\Lambda])$ are special Lagrangian. We compute the first and second variations of $\CC$ and compare with the volume functional. In particular, we show that all critical points of $\CC$ are local minima.

Furthermore, we show that $\CC$ is convex on the subspace $\L^+ \subset \L$ consisting of Lagrangian submanifolds on which $\re \Omega$ restricts to a volume form. More precisely, let $\OO^+ \subset \L^+$ be an exact compactly supported isotopy class. That is, $\OO^+$ is the set of all $\Lambda \in \L^+$ that can be connected by an exact compactly supported path to a given point in $\L^+.$ We exhibit a natural Riemannian metric on $\OO^+$ and compute the Euler-Lagrange equation of the associated energy functional. Defining geodesics to be solutions of the Euler-Lagrange equation, we show that $\CC$ restricted to any geodesic is convex.

Section~\ref{ssec:mirror} explains the analogy under mirror symmetry between $\CC$ and a functional $\MM$ introduced by Donaldson in his work~\cite{Do85} on the Kobayashi-Hitchin correspondence. Recall that Kobayashi~\cite{Ko82} proved that the existence of an Einstein-Hermitian metric on a holomorphic vector bundle implies the algebro-geometric Mumford stability condition. Kobayashi and Hitchin conjectured the converse. Donaldson~\cite{Do85,Do87} and Uhlenbeck-Yau~\cite{UY86} proved the conjecture. According to mirror symmetry, holomorphic vector bundles should be roughly analogous to Lagrangian submanifolds~\cite{Ko95}. We present evidence that the space of Hermitian metrics on a holomorphic vector bundle is analogous to an exact isotopy class $\OO^+ \subset \L^+.$ The stability of an exact isotopy class should be related to the existence of a special Lagrangian representative~\cite{Fu01,Th01,Do02}.

Donaldson's functional $\MM$ is a homotopy invariant of a path of Hermitian metrics on a holomorphic vector bundle $E.$ Critical points of $\MM$ are Einstein-Hermitian metrics. A key property of $\MM$ is its convexity along geodesics in the space of metrics. The analogy between $\CC$ and $\MM$ should already be clear. Further parallels become evident when we bring in the Yang-Mills functional on metrics on $E,$ which is analogous under mirror symmetry to the volume of a Lagrangian submanifold. The second variations of the Yang-Mills and volume functionals are both fourth-order differential operators. On the other hand, the second variations of $\CC$ and $\MM$ are both second order. That is, as noted by Donaldson, $\MM$ and $\CC$ are non-linear generalizations of Dirichlet's variational principle for the Poisson equation. The Yang-Mills and volume functionals both admit topological lower bounds. On the other hand, $\MM$ is bounded below if and only if $E$ is semi-stable. The conditions under which $\CC$ is bounded below remain to be understood.

Like Donaldson's functional, we expect the functional $\CC$ to play an important role in the study of geometric stability. However, we believe that $\CC$ must be corrected by contributions from holomorphic disks. The exact formula for such contributions will be discussed in a forthcoming paper~\cite{ST}. In fact, the considerations of that paper led us to study $\CC.$

\subsection{K\"ahler geometry}

Another analog of $\CC$ is Mabuchi's $K$-energy functional~\cite{Mab86} on the space of K\"ahler metrics representing a fixed cohomology class. The critical points of the $K$-energy functional are constant scalar curvature K\"ahler metrics, which include K\"ahler-Einstein metrics as a special case. The $K$-energy is known to be convex with respect to a natural Riemannian metric on the space of K\"ahler metrics studied by Mabuchi, Semmes and Donaldson~\cite{Mab87,Sem92,Don99a,Don02}. While it is easy to construct geodesics on the space of metrics on a holomorphic vector bundle, to construct geodesics on the space of K\"ahler metrics it is necessary to solve the homogeneous complex Monge-Amp\`ere equation. Nonetheless, many beautiful results have been obtained. For example, we refer to the work of Chen~\cite{Che00} and Chen-Tian \cite{ChT08}. It seems the questions of existence of geodesics in the space $\OO^+$ and in the space of K\"ahler metrics share important features. The author plans to address the existence question for geodesics in $\OO^+$ in future research.

In a companion paper~\cite{So13a}, we show that the Riemannian metric on $\OO^+$ has non-positive curvature. The Mabuchi-Semmes-Donaldson metric on the space of K\"ahler metrics is also known to be negatively curved. Building on the parallel with K\"ahler geometry, we outline in~\cite{So13a} a program of research towards existence criteria for special Lagrangian submanifolds based on the functional $\CC.$

\subsection{Lagrangian flux}

In Section~\ref{sec:flux} we discuss a generalization of the flux homomorphism for Lagrangian paths. The flux homomorphism first appeared in the same paper of Calabi~\cite{Ca70} that introduced the Calabi homomorphism. We recall the definition for the reader's convenience. Let $\symp(M,\omega_M)$ denote the identity component of the compactly supported symplectomorphism group of $(M,\omega_M).$ The universal cover $\widetilde\symp(M,\omega_M)$ is the space of smooth paths $\phi = \{\phi_t\}_{t\in [0,1]}$ in $\symp(M,\omega_M)$ with $\phi_0 = \id_M,$ modulo end-point preserving homotopy. Let $[\phi] \in \widetilde \symp(M,\omega_M)$ be the equivalence class of a path $\phi$ and let $\xi_t$ be the time-dependent symplectic vector field generating $\phi.$ By definition,
\[
\flux : \widetilde \symp(M,\omega) \to H^1(M,\R)
\]
is given by 
\[
\flux([\phi]) = \int_0^1 [i(\xi_t)\omega_M] dt.
\]
Calabi observed that the right hand side of the preceding equation is unchanged by an end-point preserving homotopy of $\phi$ and thus $\flux$ is well-defined. In fact, it is a homomorphism. See~\cite{Ba78,MS98} for further discussion. 

The graph of a symplectomorphism $\psi \in \symp(M,\omega_M)$ being a Lagrangian submanifold of $(M\times M,-\omega_M \boxplus \omega_M),$ it is natural to look for a generalization of $\flux$ to paths of Lagrangian submanifolds. Such a generalization of $\flux$ has be known for some time~\cite{Fu02a}. See Section~\ref{sec:flux} for the definition.

We formulate a Lagrangian analog of the flux conjecture and explore its implications. Considering $\pi_1(\symp(M,\omega_M))$ as a subgroup of $\widetilde \symp(M,\omega_M)$, let
\[
G_M = \flux(\pi_1(\symp(M,\omega_M))) \subset H^1(M,\R).
\]
Suppose that $M$ is compact. The flux conjecture asserts that $G_M$ is discrete or equivalently that $\ham(M,\omega_M)$ is closed in $\symp(M,\omega_M)$ in the $C^1$ topology~\cite{MS98}. It was proven in full generality by Ono~\cite{On06}. Fix $\Lambda_* \in \L(X,L,d)$ and let $\L_*$ be its path-connected component. Using Lagrangian flux, we define a subgroup group $G_* \subset H^1(\Lambda_*,\R)$ analogous to $G_M.$ If $G_*$ is discrete, we show the orbit of $\Lambda_*$ under $\ham(X,\omega)$ is closed. If in addition the restriction map $H^1(X) \to H^1(\Lambda_*)$ is surjective, we show that $\L_*/\ham(X,\omega)$ is Hausdorff. We compare these implications with results of Ono~\cite{On07}.

Fukaya~\cite{Fu01} defined an orbit $\OO$ of $\ham(X,\omega)$ acting on $\L$ to be \emph{stable} if it has a Hausdorff neighborhood in the quotient space $\L/\ham(X,\omega).$ He asked whether stability of $\OO$ is equivalent to the existence of a special Lagrangian representative. In light of the above results, it would be natural to reformulate Fukaya's question in terms of Lagrangian flux and the functional $\CC.$ We leave this to future work.

Section~\ref{sec:b} reviews general background on families of diffeomorphisms and the space of Lagrangian submanifolds. The proofs of Theorems~\ref{tm:main} and~\ref{tm:exact} occupy Section~\ref{sec:p}.

\subsection{Acknowledgements}
The author would like to thank G. Tian for introducing him to the subject of geometric stability and for his constant encouragement. Also, the author would like to thank L. Polterovich, E. Shelukhin, R. Thomas and E. Witten for helpful conversations. The author was partially supported by Israel Science Foundation grant 1321/2009 and Marie Curie grant No. 239381.

\section{Background}\label{sec:b}
\subsection{Terminology, notation, and conventions}\label{ssec:tn}
Let $\Theta$ be a manifold with corners. Throughout the paper, we say a family $\{\phi_{\theta}\}_{\theta \in \Theta}$ of maps of manifolds $X\to Y$ is \emph{smooth} if there exists a smooth map
\[
\phi : \Theta \times X \to Y
\]
such that $\phi_\theta = \phi|_{\{\theta\} \times X}.$

We denote the space of differential $k$-forms on a manifold $X$ by $A^k(X).$
Let $f : X \to Y$ be a map of smooth manifolds, let $v$ be a section of $f^*TY$ and let $\rho \in A^k(Y).$ Let $\xi_1,\ldots,\xi_{k-1},$ be vector fields on $X.$ We extend the interior product to a map
\[
i_v : A^k(Y) \to A^{k-1}(X)
\]
given by
\[
(i_v \rho)(\xi_1,\ldots,\xi_{k-1})|_x = \rho_{f(x)}(v(x),df_x(\xi_1(x)),\ldots,df_x(\xi_{k-1}(x))).
\]
\begin{rem}\label{rem:car}
One advantage of the notation $i_v$ above is the following version of Cartan's formula for the Lie derivative. Let $f_t : X \to Y$ be a smooth family of maps. Let $v_t = \frac{d}{dt} f_t$ and let $\rho \in A^k(Y).$ Then
\[
\frac{d}{dt} f_t^* \rho = d i_{v_t} \rho + i_{v_t} d\rho.
\]
Homotopy invariance of de Rham cohomology follows immediately by integrating this formula over $t.$
\end{rem}

We recount the following conventions to avoid confusion with regard to signs. Let $\xi$ and $\zeta$ be vector fields on $X$, and let $f: X \to \R$. Throughout the present article we adhere to the convention
\[
[\xi,\zeta]f = \xi(\zeta f) - \zeta(\xi f). 
\]
Suppose now that $(X,\omega)$ is a symplectic manifold and let $H : X \to \R.$ The corresponding Hamiltonian vector field is defined by
\[
i_\xi \omega = dH.
\]
Let $K : X \to \R$ and let $\zeta$ be the corresponding Hamiltonian vector field. We adhere to the convention that the Poisson bracket is given by
\[
\{H,K\} = \xi K = \omega(\zeta,\xi).
\]
Thus $[\xi,\zeta]$ is the Hamiltonian vector field corresponding to $\{H,K\}.$

\subsection{Two parameter families}
Let $X$ be a smooth manifold.
\begin{df}
Let $\phi_{s,t}$ be a smooth two-parameter family of diffeomorphisms of $X.$ The families of vector fields \emph{associated} to $\phi_{s,t}$ are given by
\[
\xi_{s,t}(x) = \frac{\partial \phi_{s,t}}{\partial t}\circ\phi_{s,t}^{-1}(x) \qquad
\zeta_{s,t}(x) = \frac{\partial\phi_{s,t}}{\partial s} \circ \phi_{s,t}^{-1}(x).
\]
\end{df}
The following Lemma was proved by Banyaga~\cite[Proposition I.1.1]{Ba78}.
\begin{lm}\label{lm:ban}
Let $\xi_{s,t},\zeta_{s,t},$ be the families of vector fields on $X$ associated to a two-parameter family of diffeomorphisms $\phi_{s,t}$ of $X.$ Then
\begin{equation}\label{eq:int}
\frac{\partial \xi_{s,t}}{\partial s} - \frac{\partial \zeta_{s,t}}{\partial t} = [\xi_{s,t},\zeta_{s,t}].
\end{equation}
\end{lm}
Now let $(X,\omega)$ be a symplectic manifold.
\begin{df}
Let $\phi_{s,t}$ be a two-parameter family of Hamiltonian symplectomorphisms of $X$ and let $\xi_{s,t},\zeta_{s,t},$ be the associated families of vector fields. Families of Hamiltonian functions $H_{s,t},K_{s,t},$ on $X$ are said to be \emph{associated} to $\phi_{s,t}$ if their corresponding families of Hamiltonian vector fields are $\xi_{s,t}$ and $\zeta_{s,t}$ respectively.
\end{df}
\begin{rem}\label{rem:eah}
For any family $\phi_{s,t}$ of Hamiltonian symplectomorphisms, there exist associated families of Hamiltonians by a standard argument using the flux homomorphism. See \cite[Prop. 10.17]{MS98}.
\end{rem}

Part of the following lemma may be found in~\cite[Lemma 5.3.44]{FO09}.
\begin{lm}\label{lm:Hban}
Let $H_{s,t},K_{s,t},$ be a pair of two-parameter families of functions on $X.$ If $X$ is non-compact, assume that $H_{s,t},K_{s,t},$ have compact support. Otherwise, assume the normalization
\begin{equation}\label{eq:norm}
\int_X H_{s,t} \omega^n = \int_X K_{s,t} \omega^n = 0.
\end{equation}
If $H_{s,t},K_{s,t}$ are the Hamiltonian functions associated to a two parameter family of Hamiltonian symplectomorphisms then
\begin{equation*}
\frac{\partial H_{s,t}}{\partial s} - \frac{\partial K_{s,t}}{\partial t} = \{H_{s,t},K_{s,t}\}.
\end{equation*}
More generally, if we omit the compact support (resp. normalization) requirement, there exists a function of two variables $c(s,t)$ such that
\begin{equation}\label{eq:intc}
\frac{\partial H_{s,t}}{\partial s} - \frac{\partial K_{s,t}}{\partial t} - \{H_{s,t},K_{s,t}\} = c(s,t).
\end{equation}
\end{lm}
\begin{proof}
Lemma~\ref{lm:ban} immediately implies there exists $c(s,t)$ satisfying equation~\eqref{eq:intc}. If $X$ is non-compact, then by assumption the left-hand side of equation~\eqref{eq:intc} has compact support. Since $c(s,t)$ is constant on $X,$ it must vanish. Suppose now that $X$ is compact. The Poisson bracket of any two functions has integral zero with respect to the Liouville measure $\omega^n.$ So, taking partial derivatives of equation~\eqref{eq:norm} with respect to $s$ and $t,$ the left-hand side of equation~\eqref{eq:intc} has integral zero with respect to $\omega^n.$ Therefore, $c(s,t) = 0.$ The lemma follows.
\end{proof}

\subsection{The space of Lagrangian submanifolds}
We summarize and slightly extend the discussion of Akveld-Salamon~\cite[Section 2]{AS01}.

Let $\X = \X(X,L,d)$ denote the space of Lagrangian embeddings
\[
f : L \to X, \qquad f^* \omega = 0,
\]
that, in the case that $L$ is compact, represent $d \in H_n(X).$ If $L$ is non-compact, we impose that all $f \in \X$ agree with a given $f_0$ outside a compact subset of $L.$ Let $\G$ denote the group of orientation preserving compactly supported diffeomorphisms of $L.$ Define an action
\[
\G \times \X \to \X
\]
by
\[
(\psi,f) \mapsto f \circ \psi.
\]
The space of Lagrangian submanifolds $\L(X,L,d)$ is defined to be the quotient $\L = \X/\G.$ The equivalence class $\Lambda=[f]$ of a Lagrangian embedding $f : L \to X$ is identified with the submanifold $f(L)\subset X.$

Let $\Theta$ be a contractible manifold with corners. Let $\Lambda : \Theta \to \L$ be a map. Write $\Lambda_\theta = \Lambda(\theta).$ We say that $\Lambda$ is \emph{smooth} if there exists a smooth lifting $f$ such that the following diagram commutes:
\[
\xymatrix{&\X\ar[d] \\
\Theta \ar[ur]^{f} \ar[r]^{\Lambda} & \L
}
\]
Let $\Lambda: (-\epsilon,\epsilon) \to \L$ be a smooth path with lifting $f : (-\epsilon,\epsilon)\to \X.$ Let $v_t$ be the vector field along $f_t$ defined by
\[
v_t = \frac{d}{dt} f_t.
\]
Write
\[
\alpha_t = (f_t)_*i_{v_t}\omega \in A^1(\Lambda_t).
\]
By Remark~\ref{rem:car},
\[
d\alpha_t = (f_t)_*\frac{d}{dt} f_t^* \omega = 0.
\]
The following lemma is due to Akveld and Salamon~\cite[Lemma 2.1]{AS01}.
\begin{lm}\label{lm:tan}
The $1$-form $\alpha_t$ is independent of the choice of lifting $f_t$ of $\Lambda_t.$ So, for $\Gamma \in \L$ there is a canonical isomorphism
\[
T_\Gamma \L \overset{\sim}{\longrightarrow} \{\rho \in A^1(\Gamma)|d\rho = 0\}
\]
sending the equivalence class of a smooth path $\Lambda : (-\epsilon,\epsilon)\to \L$ with $\Lambda_0 = \Gamma$ to $\alpha_0 \in A^1(\Gamma).$
\end{lm}

Lemma~\ref{lm:tan} justifies the notation
\[
\frac{d}{dt} \Lambda_t = \alpha_t.
\]
The following lemma gives a converse to Lemma~\ref{lm:tan}.
\begin{lm}\label{lm:lift}
Let $\Lambda : [0,1] \to \L$ be a smooth path and let $f$ be a lifting of $\Lambda$ to $\X.$ Let $w_t$ be a smooth family of vector fields along $f$ such that $(f_t)_*i_{w_t}\omega = \frac{d}{dt}\Lambda_t.$ Then there exists $\psi : [0,1] \to \G$ such that $F_t = f_t \circ \psi_t$ satisfies $\frac{d}{dt}F_t = w_t.$
\end{lm}
\begin{proof}
Set $v_t = \frac{d}{dt}f_t.$ Since $\Lambda_t$ is Lagrangian, it follows that $w_t - v_t$ is tangent to $\Lambda_t.$ So, there exists a unique vector field $u_t$ on $L$ such that $(f_t)_*u_t = w_t - v_t.$ We take $\psi_t$ to be the flow of $u_t.$
\end{proof}
More generally, let $\Lambda: \Theta \to \L$ be a smooth family. Let $\nu \in T_\theta \Theta$ be represented by a smooth path $\sigma : (-\epsilon,\epsilon) \to \Theta$ with $\sigma(0) = \theta.$ We define the derivative of $\Lambda$ in the direction $\nu$ by
\[
d\Lambda_\theta(\nu) = \left. \frac{d}{dt} \Lambda_{\sigma(t)}\right|_{t = 0} \in A^1(\Lambda_\theta).
\]
\begin{df}
We say the family $\Lambda$ is \emph{exact} if for all $\theta \in \Theta$ and $\nu \in T_\theta\Theta$ the derivative $d\Lambda_\theta(\nu)$ is exact. That is, there exists $h : \Lambda_\theta \to \R$ such that
\begin{equation}\label{eq:h}
d\Lambda_\theta(\nu) = dh.
\end{equation}
An exact family $\Lambda$ is \emph{compactly supported} if for all $\theta$ and $\nu$ the function $h$ in equation~\eqref{eq:h} can be chosen to have compact support.
\end{df}
The following Lemma is the same as~\cite[Lemma 2.2]{AS01}.
\begin{lm}\label{lm:un}
Let $\Lambda : [a,b] \to \L$ be a smooth path. Let $\phi : [a,b] \to \ham(X,\omega)$ be a smooth path with $\phi_0 = \id_X,$ generated by the time dependent Hamiltonian $H_t.$ Then $\phi_t(\Lambda_0) = \Lambda_t$ if and only if
\[
\frac{d}{dt} \Lambda_t = dH_t|_{\Lambda_t}
\]
for all $t.$
\end{lm}

Let $(\Theta,\theta_0)$ be a pointed contractible manifold. Let $\kappa : [0,1] \times \Theta \to \Theta$ be a piecewise smooth homotopy between $\id_\Theta$ and the constant map to $\theta_0.$ That is, $\kappa$ is continuous and there exist $0 = t_0 \leq t_1 \leq \cdots \leq t_k=1$ such that $\kappa|_{[t_i,t_{i+1}]\times\Theta}$ is smooth for $i = 0,\ldots,k-1.$ Write $\kappa_t = \kappa|_{\{t\} \times \Theta}.$ Let $v_t = \frac{d}{dt}\kappa_t.$ Let $\Lambda : \Theta \to \L$ be a smooth family. We say that $\Lambda$ is \emph{$\kappa$-exact} if the path $t \mapsto \Lambda\circ \kappa_t(\theta)$ is exact for all $\theta.$ We say that $\Lambda$ is \emph{exactly contractible} if there exists $\kappa$ such that $\Lambda$ is $\kappa$-exact. The following lemma is a slight generalization of~\cite[Lemma 2.3]{AS01}.
\begin{lm}\label{lm:elfh}
Let $(\Theta,\theta_0)$ be a pointed contractible manifold with corners. Let $\Lambda : \Theta \to \L$ be smooth and compactly supported and let $f : \Theta \to \X$ be a lifting. The following are equivalent:
\begin{enumerate}
\item\label{it:e}
$\Lambda$ is exact.
\item\label{it:ec}
$\Lambda$ is exactly contractible.
\item\label{it:h}
There exists a smooth family $\phi : \Theta \to \ham(X,\omega)$ such that $\phi_\theta(\Lambda_{\theta_0}) = \Lambda_\theta$ for all $\theta \in \Theta$ and $\phi_{\theta_0} = \id_X.$
\item\label{it:f}
There exists a smooth family $\phi : \Theta \to \ham(X,\omega)$ such that $\phi_\theta \circ f_{\theta_0} = f_\theta.$
\end{enumerate}
\end{lm}
\begin{proof}
The implication \ref{it:e} $\Rightarrow$ \ref{it:ec} is immediate.
We prove \ref{it:ec} $\Rightarrow$ \ref{it:h} as follows. Let
\[
\kappa :[0,1]\times \Theta \to \Theta
\]
be a piecewise smooth homotopy between $\id_\Theta$ and the constant map to $\theta_0$ such that $\Lambda$ is $\kappa$-exact. Let $w_t = \frac{d}{dt}
\kappa_t.$ Let $f$ be a lifting of $\Lambda$ to $\X.$ Let $h_{\theta,t}$ for $t \in [0,1]$ and $\theta \in \Theta$ be a smooth family of compactly supported functions on $L$ such that
\[
dh_{\theta,t} = f_{\kappa_t(\theta)}^* \left(d\Lambda_{\kappa_t(\theta)}(w_t)\right).
\]
The map
\[
[0,1]\times \Theta \times L \longrightarrow [0,1]\times \Theta \times X
\]
given by $(t,\theta,p) \mapsto (t,\theta,f_{\kappa_t(\theta)}(p))$ is clearly an injective immersion. Recall that a smooth compactly supported function on a submanifold can be extended to a smooth compactly supported function on the ambient manifold. Thus, we construct a piecewise smooth family of compactly supported functions $H_{\theta,t} : X \to \R$ such that $H_{\theta,t} \circ f_{\kappa_t(\theta)} = h_{\theta,t}$. Take $\phi_{\theta,t}$ to be the Hamiltonian flow of the time dependent Hamiltonian $H_{\theta,t}.$ By Lemma~\ref{lm:un}, we have
\[
\phi_{\theta,t}(\Lambda_{\theta_0}) = \Lambda_{\kappa_t(\theta)}.
\]
So, we may take $\phi_{\theta} = \phi_{\theta,1}.$

Next, we prove \ref{it:h} $\Rightarrow$ \ref{it:e}. Let $\nu \in T_\theta\Theta$ be represented by a smooth path $\sigma : (-\epsilon,\epsilon) \to \Theta$ with $\sigma(0) = \theta.$ Let $\phi_{t} = \phi_{\sigma(t)}.$ A standard argument using the flux homomorphism shows there exists a time dependent Hamiltonian $H_t : X \to \R$ generating the path $t \mapsto \phi_t.$ See \cite[Prop. 10.17]{MS98}. Let $\Lambda_t = \Lambda_{\sigma(t)}.$ By assumption $\phi_t(\Lambda_0) = \Lambda_t.$ So, using Lemma~\ref{lm:un}, we have
\begin{equation*}
d\Lambda_\theta(\nu) = \left.\frac{d}{dt} \Lambda_{t}\right|_{t = 0}  = dH_0|_{\Lambda_\theta}.
\end{equation*}
Since $\theta$ and $\nu$ were arbitrary, the desired implication follows.

The implication~\ref{it:f} $\Rightarrow$ \ref{it:h} is immediate, so it remains to prove the converse. Indeed, let $\phi_{\theta}$ be as in~\ref{it:h}. Then
\[
\psi_\theta = \phi_\theta^{-1} \circ f_\theta \circ f_{\theta_0}^{-1} : \Lambda_{\theta_0} \longrightarrow \Lambda_{\theta_0}
\]
is a family of diffeomorphisms with $\psi_{\theta_0} = \id_{\Lambda_{\theta_0}}.$ It is well-known that $\psi_{\theta}$ can be extended to a family $\tilde \psi_\theta \in \ham(T^*\Lambda_{\theta_0},\omega_{\text{can}}).$ See~\cite[Exercise 3.12]{MS98}. Weinstein's Lagrangian neighborhood theorem~\cite[Theorem~3.33]{MS98} implies that we can regard $\tilde \psi_\theta$ as a smooth family in $\ham(X,\omega).$ But then
\[
\phi_\theta \circ \tilde\psi_\theta\circ f_{\theta_0} = f_\theta.
\]
Replacing $\phi_\theta$ with $\phi_\theta \circ \tilde \psi_\theta,$ we obtain~\ref{it:f}.
\end{proof}

\begin{cy}\label{cy:help}
Suppose $\Lambda = \{\Lambda_{s,}\}_{s \in [0,1]}$ is a homotopy of compactly supported exact Lagrangian paths $\Lambda_{s,} = \{\Lambda_{s,t}\}_{t \in [0,1]}$ with $\Lambda_{s,0} = \Lambda_0$ fixed. Then $\Lambda$ considered as a family $[0,1]^2 \to \L$ is exact.
\end{cy}
\begin{proof}
Let $\kappa_r : [0,1]^2  \to [0,1]^2$ be given by
\[
\kappa_r(s,t) =
\begin{cases}
(2rs,0), & 0 \leq r \leq 1/2, \\
(s,2(r-1/2)t), & 1/2 < r \leq 1.
\end{cases}
\]
Then $\Lambda$ is $\kappa$-exact, so the corollary follows from Lemma~\ref{lm:elfh}.
\end{proof}

\section{Proofs of the main theorems}\label{sec:p}
\subsection{The general case}
\begin{lm}\label{lm:ob}
Let $\Lambda \subset X$ be Lagrangian. Let $H,K : X \to \R$ with $H|_\Lambda,K|_\Lambda,$ compactly supported, and let $\xi,\zeta,$ denote the corresponding Hamiltonian vector fields. Let $\beta$ be an $n$-form on $X$ such that $\omega \wedge \beta = 0.$ Then
\[
\int_\Lambda \{H,K\}\beta = \int_\Lambda \left(H di_{\zeta}\beta - K di_{\xi} \beta\right).
\]
\end{lm}
\begin{proof}
We have
\begin{align*}
0 &= i_{\xi}i_{\zeta}(\omega \wedge \beta) \\
&= \{H,K\} \beta - dK \wedge i_{\xi} \beta + dH\wedge i_{\zeta}\beta + \omega \wedge i_{\xi}i_{\zeta} \beta.
\end{align*}
Using the assumption that $\Lambda$ is Lagrangian, we have
\[
\{H,K\} \beta|_\Lambda = -\left(dH \wedge i_{\zeta} \beta - dK \wedge i_{\xi} \beta\right)|_\Lambda.
\]
The lemma follows by integrating by parts.
\end{proof}

Let $\OO \subset \L$ be an exact compactly supported isotopy class, or equivalently, an orbit of $\ham(X,\omega).$ For $\Lambda \in \OO$, the tangent space $T_\Lambda \OO$ is the space of exact $1$-forms on $\Lambda$ with compactly supported primitive. Let $\Upsilon$ be the $1$-form on $\OO$ given by the linear functional $dh \mapsto \int_{\Lambda} h \beta$ on the tangent space $T_\Lambda \OO.$ It is convenient to think of the functional $\CC$ as the integral along a path in $\OO$ of the $1$-form $\Upsilon.$ Thus to prove Theorem~\ref{tm:main}, it suffices to show $\Upsilon$ is closed. To this end, we compute the exterior derivative of its pull-back by an arbitrary two-parameter family.

In the following, take $\Theta = [0,1]^2$ and let $\Lambda : \Theta \to \L(X,L,d)$ be compactly supported and exact. Let $\beta$ be a closed $n$-form with $\omega \wedge \beta = 0$ and, if $L$ is compact, $\int_d \beta = 0.$ For $\theta \in \Theta$ and $\nu \in T_\theta\Theta,$ let $h_{\theta,\nu}$ be a compactly supported function on $\Lambda_{\theta}$ such that $d_\theta\Lambda(\nu) = dh_{\theta,\nu}.$ Let $\eta$ be the $1$-form on $\Theta$ given by
\[
\eta_\theta(\nu) = \int_{\Lambda_\theta} h_{\theta,\nu} \beta.
\]
If $L$ is non-compact, then $h_{\nu,\theta}$ is unique, so $\eta$ is well-defined. If $L$ is compact, then $h_{\theta,\nu}$ is unique up to a constant, so $\eta$ is well defined because of the assumption $\int_d \beta = 0.$

\begin{lm}\label{lm:closed}
The $1$-form $\eta$ is closed.
\end{lm}
\begin{proof}
By Lemma~\ref{lm:elfh}, we choose $\phi = \{\phi_{s,t}\}_{s,t \in [0,1]}$ a family of Hamiltonian symplectomorphisms of $X$ such that $\phi_{s,t}(\Lambda_0) = \Lambda_{s,t}.$ Let $\xi_{s,t},\zeta_{s,t},$ be the associated vector fields. If $X$ is non-compact, let $H_{s,t}$ and $K_{s,t}$ be the associated Hamiltonian functions with compact support. If $X$ is compact, let $H_{s,t}$ and $K_{s,t}$ be the associated Hamiltonian functions satisfying normalization~\eqref{eq:norm}. Such $H_{s,t},K_{s,t},$ exist by Remark~\ref{rem:eah}.

By Lemma~\ref{lm:un} we have
\[
dH_{s,t}|_{\Lambda_{s,t}} = dh_{(s,t),{\partial_t}}, \qquad dK_{s,t}|_{\Lambda_{s,t}} = dk_{(s,t),\partial_s}.
\]
Therefore,
\begin{equation*}
\eta = \left(\int_{\Lambda_0}\phi_{s,t}^*(H_{s,t}\beta)\right)dt + \left(\int_{\Lambda_0} \phi_{s,t}^*(K_{s,t}\beta)\right)ds.
\end{equation*}
Let $\eta_t,\eta_s,$ be the coefficients of $dt,ds,$ in the preceding expression so that $\eta = \eta_t dt + \eta_s ds.$ Using the assumption that $\beta$ is closed, we have
\begin{align}
\frac{\partial \eta_t}{\partial s}
&= \int_{\Lambda_0}\phi_{s,t}^*\left[\left(\frac{\partial H_{s,t}}{\partial s} + \{K_{s,t},H_{s,t}\}\right) \beta + H_{s,t}di_{\zeta_{s,t}}\beta\right], \label{eq:etas}\\
\frac{\partial \eta_s}{\partial t}
&= \int_{\Lambda_0}\phi_{s,t}^*\left[\left(\frac{\partial K_{s,t}}{\partial t} + \{H_{s,t},K_{s,t}\}\right)\beta + K_{s,t} di_{\xi_{s,t}}\beta\right].\label{eq:etat}
\end{align}
Subtracting equation~\eqref{eq:etat} from equation~\eqref{eq:etas} and using Lemmas~\ref{lm:Hban} and~\ref{lm:ob}, we deduce that $d\eta = 0.$
\end{proof}
The following proposition gives the first variational formula for $\CC$, and as we see below, immediately implies Theorem~\ref{tm:main}.
\begin{pr}\label{pr:fv}
Let $\{\Lambda_{s,}\}_{s\in [0,1]}$ be a smooth family of exact compactly supported Lagrangian paths $\Lambda_{s,} = \{\Lambda_{s,t}\}_{t \in [0,1]}$ with $\Lambda_{s,0} = \Lambda_0$ fixed. Let $k_s:\Lambda_{s,1} \to \R$ be a family of functions with compact support such that $\frac{d}{ds} \Lambda_{s,1} = dk_{s}.$ Then
\[
\frac{d}{ds} \CC(\Lambda_{s,}) = \int_{\Lambda_{s,1}}k_s\beta.
\]
\end{pr}
\begin{proof}
By Corollary~\ref{cy:help}, $\Lambda$ considered as a family $\Theta = [0,1]^2 \to \L$ is exact. So, we take $\eta$ as above. Moreover, we claim that
\begin{equation}\label{eq:e0}
\eta|_{[0,1] \times \{0\}} = 0.
\end{equation}
Indeed, since $\Lambda_{s,0} = \Lambda_0$ is fixed, we may take $h_{(s,0),\partial_s} = 0.$ By equation~\eqref{eq:e0}, Lemma~\ref{lm:closed} and Stokes' theorem, we have
\begin{equation*}
\CC(\Lambda_{s,1}) - \CC(\Lambda_{0,1}) = \int_{[0,s]\times\{1\}} \eta + \int_{\partial ([0,s] \times [0,1] )} \eta = \int_{[0,s]\times\{1\}} \eta.
\end{equation*}
So,
\[
\frac{\partial}{\partial s} \CC(\Lambda_{s,1}) = \frac{\partial}{\partial s} \int_{[0,s]\times\{1\}} \eta = \eta_{(s,1)}(\partial_s) = \int_{\Lambda_{s,1}}k_s \beta.
\]
\end{proof}

\begin{proof}[Proof of Theorem~\ref{tm:main}]
Suppose $\Lambda = \{\Lambda_{s,}\}_{s \in [0,1]}$ is a homotopy of compactly supported exact Lagrangian paths $\Lambda_{s,} = \{\Lambda_{s,t}\}_{t \in [0,1]}$ with $\Lambda_{s,i} = \Lambda_i$ fixed for $i = 0,1.$ Since $\frac{d}{ds}\Lambda_{s,1} = 0,$ we apply Proposition~\ref{pr:fv} with $k_s = 0$ to conclude that $\CC(\Lambda_{s,})$ is independent of $s.$
\end{proof}
\subsection{The exact case}
\begin{proof}[Proof of Theorem~\ref{tm:exact}]
Let $\Lambda = \{\Lambda_t\}_{t\in [0,1]}$ be an exact compactly supported path of Lagrangian submanifolds. By Lemma~\ref{lm:elfh}, we choose a family $\{\phi_t\}_{t \in [0,1]}$ in $\ham(X,\omega)$ such that $\phi_t(\Lambda_0) = \Lambda_t.$ Let $\xi_t$ be the associated family of vector fields, and let $H_t$ be the associated compactly supported time-dependent Hamiltonian function. Using Cartan's formula and integration by parts, we calculate
\begin{align}
\frac{d}{dt} \int_{\Lambda_t} \lambda\wedge \gamma &= \frac{d}{dt} \int_{\Lambda_0} \phi_t^*(\lambda\wedge\gamma) \label{eq:dt}\\
& = \int_{\Lambda_0}\phi_t^*\left[ d(i_{\xi_t}\lambda + H_t)\wedge \gamma + \lambda \wedge (di_{\xi_t}\gamma + i_{\xi_t}\beta) \right] \notag\\
& = -\int_{\Lambda_0}\phi_t^*\left(H_t \beta + i_{\xi_t}\lambda \wedge \beta - \lambda\wedge i_{\xi_t}\beta + \omega \wedge i_{\xi_t}\gamma\right).\notag
\end{align}
Using the fact that $\phi_t^*\omega = \omega$ and $\Lambda_0$ is Lagrangian, we have
\begin{equation}\label{eq:lag}
\int_{\Lambda_0}\phi_t^*(\omega \wedge i_{\xi_t}\gamma) = \int_{\Lambda_0} \omega \wedge \phi_t^*(i_{\xi_t}\gamma) = 0.
\end{equation}
Using the derivation property of $i_{\xi_t}$ and assumption~\eqref{eq:lbog}, we calculate
\begin{align*}
&i_{\xi_t}\lambda \wedge \beta - \lambda\wedge i_{\xi_t}\beta = i_{\xi_t}(\lambda\wedge\beta) = \\
& \qquad \qquad \qquad = -c i_{\xi_t}(\omega \wedge \gamma) = -c \left(dH_t \wedge \gamma + \omega \wedge i_{\xi_t}\gamma\right).
\end{align*}
So, by equation~\eqref{eq:lag} and integration by parts, we have
\begin{align}\label{eq:iixlb}
\int_{\Lambda_0}\phi_t^*\left(i_{\xi_t}\lambda \wedge \beta - \lambda\wedge i_{\xi_t}\beta\right) &= -c\int_{\Lambda_0}\phi_t^*(dH_t \wedge \gamma) \\
& = c\int_{\Lambda_0}\phi_t^*(H_t \beta).\notag
\end{align}
Combining equations~\eqref{eq:dt},~\eqref{eq:lag},~\eqref{eq:iixlb} and Lemma~\ref{lm:un}, we obtain
\[
\frac{d}{dt} \int_{\Lambda_t} \lambda\wedge \gamma = -(1+c)\int_{\Lambda_t} H_t \beta = -(1+c)\int_{\Lambda_t} h_t \beta.
\]
The theorem follows by integrating over $t \in [0,1].$
\end{proof}

\section{The Calabi homomorphism}\label{sec:class}
Let $(M,\omega_M)$ be a non-compact symplectic manifold of dimension $2m.$ Let $X = M \times M$ and let $p_i : X \to M$ for $i = 1,2,$ denote the projections to the first and second factors respectively. Let $\omega = -p_1^*\omega_M + p_2^* \omega_M.$ Let
\[
\beta = \frac{1}{m+1}\sum_{i = 0}^m p_1^*\omega_M^{i} \wedge p_2^*\omega_M^{m-i}.
\]
Then
\[
\omega \wedge \beta = \frac{1}{m+1}\left(-p_1^*\omega_M^{m+1} + p_2^*\omega_M^{ m+1}\right) = 0
\]
because $\dim M = 2m < 2(m+1).$
Suppose $\phi = \{\phi_t\}_{t \in [0,1]}$ is a path in $\ham(M,\omega_M)$ and $\Lambda = \{\Lambda_t\}_{t \in [0,1]}$ is the corresponding Lagrangian path in $\L(X,M,\cdot).$ Let $\{H_t\}_{t\in[0,1]}$ be the time dependent Hamiltonian function generating $\phi,$ and set
\[
h_t = H_t \circ p_2|_{\Lambda_t}.
\]
Then
\[
\frac{d}{dt} \Lambda_t = dh_t.
\]
With the above choice of $\beta,$ we have
\begin{equation*}
\CC([\Lambda]) = \int_0^1 \int_{\Lambda_t} h_t \beta = \int_0^1 \int_{M}H_t \omega^n = \cal([\phi]).
\end{equation*}
In particular, Theorem~\ref{tm:main} implies the known result that $\cal$ is well-defined.

We turn to the exact case $\omega_M = d\lambda_M.$ Denote by $\xi_M$ the Liouville vector field of $\lambda_M.$ Let $\lambda = -p_1^*\lambda_M + p_2^*\lambda_M.$ So, $d\lambda = \omega$ and the Liouville vector field of $\lambda$ is the unique $\xi$ such that $dp_i \circ \xi = (-1)^i\xi_M \circ p_i$ for $i = 1,2.$ Moreover, we have $L_\xi \beta = m \beta.$ Following Example~\ref{ex:exact}, take
\begin{align*}
\gamma &= \frac{1}{m}i_\xi \beta \\
&= \frac{1}{m(m+1)}\sum_{i = 0}^{m-1} p_1^* \omega_M^i \wedge p_2^*\omega_M^{m-i-1} \wedge \left((i+1) p_1^*\lambda_M + (m-i) p_2^*\lambda_M\right).
\end{align*}
Thus,
\begin{equation}\label{eq:lg}
\lambda \wedge \gamma = -\frac{1}{m} p_1^* \lambda_M \wedge p_2^*\lambda_M \wedge \sum_{i=0}^{m-1} p_1^* \omega_M^i \wedge p_2^*\omega_M^{m-i-1}.
\end{equation}
Let $\phi$ and $\Lambda$ be as above.
It follows from Theorem~\ref{tm:exact} and equation~\eqref{eq:lg} that
\begin{align*}
\cal([\phi]) = \CC([\Lambda]) &= -\frac{1}{m+1}\int_{\Lambda_1} \lambda \wedge \gamma \\
&= -\frac{1}{m+1}\int_M \phi_1^*\lambda_M \wedge \lambda_M\wedge \omega_M^{m-1},
\end{align*}
a formula due to Banyaga~\cite{Ba78}.

\section{Special Lagrangians}\label{sec:slag}
\subsection{Background}
Before explaining how the functional $\CC$ is related to special Lagrangian submanifolds, we recall several definitions and their implications. The following definition is from~\cite{Jo03} with the slight modification that we do not require compactness.
\begin{df}
An $n$-dimensional \emph{almost Calabi-Yau} manifold is a quadruple $(X,J,\omega,\Omega),$ where $(X,J)$ is an $n$-dimensional complex manifold, $\omega$ is a K\"ahler form on $M,$ and $\Omega$ is a nowhere vanishing holomorphic $(n,0)$-form.

We call $(X,J,\omega,\Omega)$ \emph{Calabi-Yau} if in addition
\begin{equation}\label{eq:cy}
\omega^n/n! = (-1)^{n(n-1)/2}(\sqrt{-1}/2)^n\Omega \wedge \overline\Omega.
\end{equation}
\end{df}

We explain briefly what it means to be an (almost) Calabi-Yau manifold. A K\"ahler manifold $(X,J,\omega)$ has a holomorphically trivial canonical bundle if and only if it admits a non-zero holomorphic $(n,0)$ form, which makes it into an almost Calabi-Yau manifold. The existence of a holomorphic $(n,0)$ form $\Omega$ satisfying condition~\eqref{eq:cy} implies the Ricci curvature of $\omega$ vanishes. On the other hand, if $X$ is compact and the real first Chern class of $X$ satisfies $c_1(TX) = 0,$ there exists a unique $\omega$ with vanishing Ricci curvature in any K\"ahler class on $X$ by results of Calabi~\cite{Ca57} and Yau~\cite{Ya78}. If in addition the canonical bundle of $X$ is trivial, there exists a holomorphic $(n,0)$ satisfying condition~\eqref{eq:cy}. See~\cite{Jo08} for further discussion.

The following ideas are due to Harvey and Lawson~\cite{HL82}.
\begin{df}
A closed $p$-form $\chi$ on a Riemannian manifold $(X,g)$ is a \emph{calibration} if for all oriented tangent $p$-planes $\tau$ there holds
\[
\chi|_\tau \leq \vol_\tau.
\]
A submanifold $N \subset X$ is \emph{calibrated} with respect to $\chi$ if
\[
\chi|_N = \vol_N.
\]
\end{df}
Compact calibrated submanifolds minimize volume in the their homology class. Non-compact calibrated submanifolds are locally volume minimizing. Moreover, calibrated submanifolds have a canonical orientation.

Let $(X,J,\omega,\Omega)$ be an almost Calabi-Yau $n$-fold, and let $\rho : X \to \R$ be the smooth function satisfying
\begin{equation}\label{eq:acy}
e^\rho \omega^n/n! = (-1)^{n(n-1)/2}(\sqrt{-1}/2)^n \Omega \wedge \overline \Omega.
\end{equation}
Let $g,\tilde g,$ be the Riemannian metrics on $X$ defined by
\[
g(\xi,\zeta) = \omega(\xi,J\zeta), \qquad \tilde g(\xi,\zeta) = e^{\rho/n}g(\xi,\zeta).
\]
Namely, $g$ is the K\"ahler metric associated with $\omega$ and $\tilde g$ is a conformal rescaling. If $(X,J,\omega,\Omega)$ is Calabi-Yau, then $\rho = 0,$ so $\tilde g = g.$

Harvey and Lawson showed that for any tangent $n$-plane $\tau,$
\begin{equation}\label{eq:sli}
\left|\rule{0pt}{11pt}\re \Omega|_\tau\right|_{\tilde g}^2 + \left|\rule{0pt}{11pt}\im \Omega|_\tau\right|_{\tilde g}^2 \leq 1,
\end{equation}
with the equality holding for $\tau$ Lagrangian. It follows that $\re \Omega$ is a calibration on $(X,\tilde g).$
\begin{df}
Let $(X,J,\omega,\Omega)$ be almost Calabi-Yau $n$-fold. An $n$-dimensional real submanifold $\Lambda \subset X$ is \emph{special Lagrangian} if
\[
\omega|_\Lambda = 0, \qquad \im \Omega|_{\Lambda} =0.
\]
\end{df}
Inequality~\eqref{eq:sli} implies that a submanifold $\Lambda \subset X$ is calibrated with respect to $\re \Omega$ if and only if it is special Lagrangian. For any Lagrangian submanifold $\Lambda \subset X,$ the equality case of~\eqref{eq:sli} implies there exists a smooth function $\vartheta: \Lambda \to S^1$ such that
\[
\Omega|_\Lambda = e^{\sqrt{-1}\vartheta}\vol_{\Lambda,\tilde g} = e^{\sqrt{-1}\vartheta + \rho/2}\vol_{\Lambda,g}.
\]
We call $\vartheta$ the \emph{phase} function of $\Lambda.$ The differential $d\vartheta$ is a well-defined closed $1$-form, which is exact if and only if $\vartheta$ lifts to $\R.$

\subsection{Variational principle}
We return to the object of the paper. Let $(X,J,\omega,\Omega)$ be an almost Calabi-Yau $n$-fold.  Below, all norms, volume forms, gradients and so on, will be those associated to the K\"ahler metric on $X,$ which we denote by $g.$ Since $\Omega$ is of type $(n,0)$ and $\omega$ is of type $(1,1),$ we have
\[
\omega \wedge \Omega = 0.
\]
So, we may apply Theorem~\ref{tm:main} with $\beta = \im \Omega.$ In fact, this is a special case of Example~\ref{ex:prim}. The requirement $\int_d \im \Omega = 0$ can always be satisfied by multiplying $\Omega$ by a complex constant of unit modulus. For the rest of this section, we take $\beta = \im \Omega.$ In the following discussion of variational formulae, we always consider variations of an exact compactly supported Lagrangian path $\Lambda = \{\Lambda_t\}_{t \in [0,1]}$ holding $\Lambda_0$ fixed.

\begin{cy}\label{cy:sl}
A path $\Lambda = \{\Lambda_t\}_{t \in [0,1]}$ is a critical point of $\CC$ if and only if $\Lambda_1$ is special Lagrangian.
\end{cy}
\begin{proof}
Let $k : \Lambda_{1} \to \R.$ By Proposition~\ref{pr:fv}, the first variation of $\CC$ in the direction $dk$  is given by $\int_{\Lambda_1} k \wedge \im \Omega$. This integral vanishes for all $k$ if and only if $\im \Omega|_{\Lambda_1} = 0.$
\end{proof}
\begin{pr}\label{pr:svc}
Let $\{\Lambda_{s,}\}_{s \in (-\epsilon,\epsilon)}$ be a family of exact compactly supported Lagrangian paths with $\Lambda_{s,0} = \Lambda_0$ fixed. Suppose $\Lambda_{0,}$ is a critical point of $\CC$ and let $k : \Lambda_{0,1} \to \R$ be such that $\frac{d}{ds}\Lambda_{s,1}|_{s=0} = dk.$ Then
\[
\left .\frac{d^2}{ds^2} \CC([\Lambda_{s,}])\right|_{s = 0} = \int_{\Lambda_{0,1}} dk \wedge i_{\nabla k} \re \Omega= \int_{\Lambda_{0,1}} |dk|^2 e^{\rho/2}\vol.
\]
So, the second variation is non-negative, vanishing if and only if $k$ is constant.
\end{pr}
\begin{proof}
Let $f_s : L \to X$ for $s \in (-\epsilon,\epsilon)$ be a smooth family lifting $\Lambda_{s,1}$ to $\X.$ By Lemma~\ref{lm:lift}, we may assume that
\[
\left.\frac{d}{ds}f_s\right|_{s = 0} = - J \nabla k\circ f_s.
\]
Let $k_s : \Lambda_{s,1} \to \R$ be a family of functions with $dk_s = \frac{d}{ds} \Lambda_{s,1}.$
Since $\Lambda_{0,}$ is $\CC$ critical, by Corollary~\ref{cy:sl} we have $\im \Omega|_{\Lambda_{0,}} = 0.$ So, by Proposition~\ref{pr:fv} we have
\begin{align*}
\left .\frac{d^2}{ds^2} \CC([\Lambda_{s,}])\right|_{s =0} &= \left .\frac{d}{ds}\right|_{s = 0} \int_{\Lambda_{s,1}} k_s \im \Omega \\
&= \left .\frac{d}{ds}\right|_{s=0}\int_L f_s^*(k_s \im \Omega) \\
& = \int_{\Lambda_{0,1}} k  di_{(-J\nabla k)} \im \Omega \\
& = -\int_{\Lambda_{0,1}} k di_{\nabla k} \re \Omega \\
& = \int_{\Lambda_{0,1}} dk \wedge i_{\nabla k} \re \Omega.
\end{align*}
Since $\Lambda_{0,1}$ is calibrated with respect to $\tilde g$ and $\re \Omega,$ we have
\[
i_{\nabla k} \re \Omega = e^{\rho/2}i_{\nabla k}\vol = e^{\rho/2}*dk.
\]
So,
\[
\int_{\Lambda_{0,1}} dk \wedge i_{\nabla k} \re \Omega = \int_{\Lambda_{0,1}} |dk|^2 e^{\rho/2}\vol.
\]
\end{proof}
\subsection{Geodesics and convexity}\label{ssec:gc}
Let
\[
\L^+ \subset \L(X,L,d)
\]
be the open subspace consisting of $\Lambda$ such that $\re \Omega|_{\Lambda}$ is nowhere-vanishing and agrees with the orientation of $\Lambda.$ In the terminology of~\cite{CL04,Ne07}, the Lagrangians in $\L^+$ are called almost calibrated. For simplicity, we restrict to the case that $L$ is compact. Let $\OO^+ \subset \L^+$ be an exact isotopy class and let $\widetilde \OO^+$ be its universal cover. We show that $\CC$ considered as a functional $\widetilde\OO^+ \to \R$ is convex. That is, the second derivative of $\CC$ along a geodesic in $\widetilde\OO^+$ is positive.

The key point here is the definition of geodesic, which in turn rests upon the definition of a connection on the principle bundle
\[
\X^+ \to \L^+
\]
with fiber
\[
\X^+_{\Lambda} = \{ f : L \to \Lambda | f^*\re \Omega = \mu\},
\]
where $\mu$ is a volume form with $\int_{L}\mu = \int_d \re \Omega.$ The structure group $\G^+$ of $\X^+$ is the $\mu$-preserving diffeomorphisms of $L.$ Moser's argument shows that $\X^+$ is non-empty. See~\cite{Do99} and~\cite{Wei90} for related constructions.

Given an exact path $\Lambda: (-\epsilon,\epsilon) \to \L^+$ and an embedding $f_0 : L \to X$ with image $\Lambda_0$ and $f_0^*\re\Omega = \mu,$ we define as follows the  horizontal lift $f$ of $\Lambda$ extending $f_0$. Suppose $h_t : \Lambda_t \to \R$ is a family of functions with $\frac{d}{dt}\Lambda_t = dh_t.$ Recall our extension of interior multiplication to vector fields along a map in Section~\ref{ssec:tn}. Let $v_t$ be the unique vector field along the inclusion map of $\Lambda_t$ such that
\begin{align}\label{eq:o}
i_{v_t}\omega = dh_t, \\
i_{v_t}\!\re\Omega = 0. \label{eq:O}
\end{align}
Condition~\eqref{eq:o} determines $v_t$ up to a vector tangent to $\Lambda_t.$ Since $\re\Omega|_{\Lambda_t}$ is assumed non-degenerate, condition~\eqref{eq:O} determines the tangential component of $v_t.$ So, using Lemma~\ref{lm:lift}, we take $f$ to be the lift extending $f_0$ and satisfying
\begin{equation}\label{eq:par}
\frac{d}{dt} f_t = v_t \circ f_t.
\end{equation}
It follows from condition~\eqref{eq:O} and Cartan's formula that $f_t^*\re \Omega$ is constant, so $f$ does in fact lift $\Lambda$ to $\X^+ \subset \X.$

\begin{rem}
We have seen another way of defining a horizontal lift $f_t$ in the proof of Proposition~\ref{pr:svc}, although in this case only to $\X$. Namely,
\begin{equation}\label{eq:naive}
\frac{d}{dt} f_t = - J \nabla h_t \circ f_t.
\end{equation}
The vector field $-J\nabla h_t$ is uniquely characterized by being perpendicular to $\Lambda_t$ and satisfying
\begin{equation}\label{eq:jht}
i_{-J\nabla h_t}\omega = dh_t.
\end{equation}
By way of comparison, denoting by $\vartheta_t$ the phase function of $\Lambda_t,$ we claim that
\begin{equation}\label{eq:vtmet}
v_t = -J\nabla h_t - \tan(\vartheta_t) \nabla h_t.
\end{equation}
Indeed, for such $v_t$ we have
\begin{align*}
i_{v_t}\re\Omega &= -i_{J\nabla h_t}\re\Omega - \tan(\vartheta_t) i_{\nabla h_t}\re\Omega \\
&= i_{\nabla h_t}\im \Omega - \sin(\vartheta)e^{\rho/2} i_{\nabla h_t} \vol \\ & = 0.
\end{align*}
The condition $i_{v_t}\omega = dh_t$ follows from equation~\eqref{eq:jht} and the fact that $\nabla h_t$ is tangent to $\Lambda_t.$
\end{rem}

For $\Lambda \in \L^+,$ let
\[
\HH_\Lambda = \left\{ h \in C^\infty(\Lambda)\left| \int_{\Lambda} h \re \Omega = 0\right.\right\}.
\]
The map $\HH_\Lambda \to T_{\Lambda}\L$ given by $h \mapsto dh$ identifies $\HH_{\Lambda}$ with the space of exact first order deformations of $\Lambda.$  Define an inner product on $\HH_\Lambda$ by
\[
(h,k) = \int_\Lambda hk \re \Omega.
\]
Let $\HH\to \L^+$ be the bundle with fiber $\HH_\Lambda.$ That is, $\HH$ is the vector bundle associated with the $\G^+$ principal bundle $\X^+$ and the representation of $\G^+$ on
\[
C_0^\infty(L) = \left\{ h \in C^\infty(L)\left| \int_L h \mu = 0 \right.\right\},
\]
given by $(\psi,f) \mapsto f \circ \psi.$ This representation preserves the $L^2$ inner product on $C^\infty_0(L)$ with respect to $\mu.$ In fact, the $L^2$ inner product on $C^\infty_0(L)$ induces the metric $(\cdot,\cdot)$ on $\HH.$
So, the connection on $\X^+$ induces a connection on $\HH,$ which preserves the metric $(\cdot,\cdot).$ Thus it is natural to make the following definitions.
\begin{df}\label{df:eg}
Let $\Lambda : [a,b] \to \L^+$ be an exact path, and let $h_t \in \H_{\Lambda_t}$ satisfy $dh_t = \frac{d}{dt}\Lambda_t.$ The \emph{energy} of $\Lambda$ is given by
\[
E(\Lambda) = \int_a^b (h_t,h_t) dt = \int_a^b \int_{\Lambda_t} h_t^2 \re \Omega.
\]
Let $f$ be a horizontal lifting of $\Lambda$ to $\X^+.$ We call $\Lambda$ a \emph{geodesic} if $h_t \circ f_t$ is constant in $t.$
\end{df}
\begin{pr}\label{pr:fvE}
An exact Lagrangian path $\{\Lambda_t\}_{t\in[a,b]}$ is a geodesic if and only if it is a critical point of the energy functional with respect to proper exact variations.
\end{pr}
\begin{proof}
Let $\{\Lambda_{s,}\}_{s\in (-\epsilon,\epsilon)}$ be a family of exact Lagrangian paths such that $\Lambda_{0,t} = \Lambda_t$ and $\Lambda_{s,i} = \Lambda_{i}$ for $i = a,b.$  Let $h_{s,t} \in \HH_{\Lambda_{s,t}}$ satisfy $dh_{s,t} = \frac{\partial}{\partial t}\Lambda_{s,t}$ and write $h_t = h_{0,t}.$ For the rest of the proof of this lemma, take $\Theta = (-\epsilon,\epsilon)\times [a,b]$ and let $\Lambda : \Theta \to \L$ denote the two-parameter family $\{\Lambda_{s,t}\}_{(s,t) \in \Theta}.$ By Corollary~\ref{cy:help} we know that $\Lambda$ is an exact family. Let $f : \Theta \to \X$ be a lifting of $\Lambda$ such that $f_{s,}$ is horizontal and write $f_t = f_{0,t}.$ By Lemma~\ref{lm:elfh}, we choose a family $\phi_{s,t} \in \ham(X,\omega)$ such that $\phi_{s,t} \circ f_{0,0} = f_{s,t}.$ Let $\xi_{s,t}$ and $\zeta_{s,t}$ be the associated vector fields. Let $H_{s,t}$ and $K_{s,t}$ be associated Hamiltonian functions with $H_{s,t}$ normalized so that
\begin{equation}\label{eq:Hnor}
\int_{\Lambda_{s,t}} H_{s,t} \re \Omega = 0.
\end{equation}
Thus, by Lemma~\ref{lm:un} we have $H_{s,t}|_{\Lambda_{s,t}} = h_{s,t}.$
So, we calculate
\begin{align}
\frac{d}{ds} E(\Lambda_{s,}) &= \frac{1}{2} \frac{d}{ds} \int_a^b \int_{L} f_{s,t}^*(H_{s,t}^2 \re \Omega) dt \notag\\
& = \int_a^b \int_L f_{s,t}^*\left[\frac{1}{2} H_{s,t}^2 di_{\zeta_{s,t}}\re\Omega\, + \right. \label{eq:dE}\\
 & \qquad\qquad \qquad \qquad  \left . +\, H_{s,t}\left(\{K_{s,t},H_{s,t}\} + \frac{\partial H_{s,t}}{\partial s}\right)\re\Omega\right] dt. \notag
\end{align}
Because we chose $f_{s,}$ to be horizontal, we have
\begin{equation}\label{eq:fhor}
i_{\xi_{s,t}}\re \Omega|_{\Lambda_{s,t}} = 0.
\end{equation}
Since $H_{s,t}\xi_{s,t}$ is the Hamiltonian vector field of $\frac{1}{2} H_{s,t}^2,$ Lemma~\ref{lm:ob} and equation~\eqref{eq:fhor} give
\begin{align}
&\frac{1}{2}\int_L f_{s,t}^*\left(H_{s,t}^2 di_{\zeta_{s,t}} \re \Omega\right) = \label{eq:dir}\\
& \qquad = \frac{1}{2}\int_{\Lambda_{s,t}} H_{s,t}^2 di_{\zeta_{s,t}}\re\Omega \notag\\
& \qquad = \int_{\Lambda_{s,t}}H_{s,t}\{H_{s,t},K_{s,t}\}\re\Omega +  K_{s,t} d\left(H_{s,t} i_{\xi_{s,t}}\re\Omega\right) \notag\\
& \qquad = \int_{\Lambda_{s,t}}H_{s,t}\{H_{s,t},K_{s,t}\}\re\Omega.\notag
\end{align}
Combining equations~\eqref{eq:dE} and~\eqref{eq:dir}, we obtain
\begin{equation*}
\frac{d}{ds} E(\Lambda_{s,}) =  \int_a^b \int_{L} f_{s,t}^*\left(H_{s,t}\frac{\partial H_{s,t}}{\partial s} \re \Omega\right) dt.
\end{equation*}
Using Lemma~\ref{lm:Hban} equation~\eqref{eq:intc} combined with normalization~\eqref{eq:Hnor}, we calculate
\begin{align*}
\int_{L} f_{s,t}^*\left(H_{s,t}\frac{\partial H_{s,t}}{\partial s} \re \Omega\right) &= \int_L f_{s,t}^*\left[H_{s,t}\left(\frac{\partial K_{s,t}}{\partial t} + \{H_{s,t},K_{s,t}\}\right)\re \Omega \right] \\
&= \frac{d}{dt}\int_{L}f_{s,t}^*(H_{s,t}K_{s,t}\re\Omega)\, - \\ &\qquad \qquad -\int_L f_{s,t}^*\left(\frac{\partial H_{s,t}}{\partial t} K_{s,t}\re \Omega \right).
\end{align*}
In the second transition, we have used equation~\eqref{eq:fhor}. For $i = a,b,$ since $\Lambda_{s,i} = \Lambda_i$ is fixed, we deduce that $K_{s,i}|_{\Lambda_{s,i}}$ is constant. So, by normalization~\eqref{eq:Hnor} the boundary contributions from integration by parts vanish, and we are left with
\[
\frac{d}{ds} E(\Lambda_{s,}) = -\int_a^b \int_{L}  \frac{\partial}{\partial t}(h_{s,t}\circ f_{s,t}) f_{s,t}^*( K_{s,t}\re \Omega)dt.
\]
It follows that $\{\Lambda_t\}$ is a critical point of $E$ if and only if $h_t \circ f_t$ is constant.
\end{proof}
\begin{pr}\label{pr:con}
Let $\Lambda : [a,b]\times [0,1] \to \L$ be a family of exact paths such that $\Lambda_{s,0} = \Lambda_0$ is fixed and $\Lambda_{,1} = \{\Lambda_{s,1}\}_{s \in [0,1]}$ is a geodesic. Let $k_s \in \HH_{\Lambda_{s,1}}$ satisfy $\frac{d}{ds}\Lambda_{s,1} = dk_s$ and let $\vartheta_s$ be the phase of $\Lambda_{s,1}.$ Then
\[
\frac{d^2}{ds^2}\CC([\Lambda_{s,}])= \int_{\Lambda_{s,1}}  \frac{|dk_s|^2}{\cos(\vartheta_s)}e^{\rho/2}\vol.
\]
In particular, if $\Lambda_{,1}$ is not constant, then
\[
\frac{d^2}{ds^2}\CC([\Lambda_{s,}]) > 0.
\]
\end{pr}
\begin{proof}
Let $v_s$ be the family of vector fields along $\Lambda_{s,1}$ satisfying conditions~\eqref{eq:o} and~\eqref{eq:O} with $s$ in place of $t$ and $k$ in place of $h.$ Let $f$ be a horizontal lifting of $\Lambda_{,1}.$ Since $\Lambda_{,1}$ is a geodesic, $k_s \circ f_s$ is independent of $s.$ So, we write $k = k_s \circ f_s.$
By equation~\eqref{eq:vtmet} with $s$ in place of $t$ and $k$ in place of $h,$ we have
\begin{align*}
i_{v_s}\im \Omega &= i_{(-J\nabla k_s)} \im \Omega - \tan(\vartheta_t)  i_{\nabla k_s} \im\Omega \\
& = -i_{(\nabla k_s)} \re \Omega - \tan(\vartheta_t)  i_{\nabla k_s} \im\Omega\\
& = -\left(\cos(\vartheta_s) + \frac{\sin^2(\vartheta_s)}{\cos(\vartheta_s)}\right) e^{\rho/2}i_{\nabla k_s}\vol \\
& = -\frac{e^{\rho/2}i_{\nabla k_s}\vol}{\cos(\vartheta)}.
\end{align*}
So,
\begin{align*}
\frac{d^2}{ds^2} \CC([\Lambda_{s,}]) &= \frac{d}{ds} \int_{\Lambda_{s,1}} k_s \im \Omega \\
&= \frac{d}{ds}\int_L k f_s^*(\im \Omega) \\
& = \int_{\Lambda_{s,1}} k_s  d i_{v_s} \im \Omega \\
& = -\int_{\Lambda_{s,1}} k_s d \left(\frac{e^{\rho/2}i_{\nabla k_s}\vol}{\cos(\vartheta_s)}\right) \\
& = \int_{\Lambda_{s,1}} \frac{|dk_s|^2}{\cos(\vartheta_s)} e^{\rho/2}\vol.
\end{align*}
Since $k = k_s \circ f_s$ is independent of $s,$ if $dk_s = 0$ for any $s,$ then $dk = 0$ and $\Lambda_{s,1}$ is the constant path. The final claim of the proposition follows.
\end{proof}

It is not hard to write down geodesics in $\L^+$ explicitly for simple examples of $(X,J,\omega,\Omega),L,d.$ The author plans to address the general existence problem for geodesics in a future paper.

\subsection{The volume functional}
It is interesting to compare $\CC$ with the volume functional,
\[
\Vol : \L \longrightarrow \R.
\]
For simplicity, we assume in the following that $L$ is compact and $(X,J,\omega,\Omega)$ is Calabi-Yau. Since special Lagrangians are calibrated, they are global minima of $\Vol.$ In fact, equation~\eqref{eq:sli} implies that for $\Lambda \in \L,$ we have the topological lower bound
\begin{equation}\label{eq:mbog}
\Vol(\Lambda) \geq \int_\Lambda \re \Omega = \int_d \re \Omega
\end{equation}
with equality exactly when $\Lambda$ is special.

We proceed to the variational formulae for $\Vol.$ Let $\Lambda = \{\Lambda_t\}$ be an exact Lagrangian path, and let $h_t : \Lambda_t \to \R$ be such that $dh_t = \frac{d}{dt}\Lambda_t.$ Let $\vartheta_t$ be the phase function of $\Lambda_t.$ Let $\langle \cdot,\cdot\rangle_t$ be the induced metric on $\Lambda_t,$ and let $\vol_t$ and $\Delta_t$ be the associated volume form and Laplacian. When clear from the context, we drop the subscript $t.$ Let $f$ be the lifting of $\Lambda$ to $\X$ such that $\frac{d}{dt}f_t = -J\nabla h_t,$ which exists by Lemma~\ref{lm:lift}. The following lemma can be found in~\cite{TY02}.
\begin{lm}
We have
\begin{gather}
\frac{d}{dt} \vartheta_t\circ f_t = (\Delta h_t)\circ f_t, \label{eq:vth}\\
\frac{d}{dt} f_t^*\vol = f_t^*\left(\langle dh_t,d\vartheta_t\rangle \vol\right). \notag
\end{gather}
\end{lm}
The following corollary is immediate.
\begin{cy}\label{cy:fvv}
The first variation of $\Vol$ is given by
\[
\frac{d}{dt} \Vol(\Lambda_t) = \int_{\Lambda_t} \langle dh_t,d\vartheta_t\rangle \vol = \int_{\Lambda_t} h_t\, d^*\!d\vartheta_t \vol.
\]
In particular, $\Lambda$ is a critical point of $\Vol$ if and only if $d\vartheta$ is harmonic.
\end{cy}

\begin{pr}\label{pr:svv}
Suppose that $\Lambda_0$ is a critical point of $\Vol.$ Then
\[
\left.\frac{d^2}{dt^2} \Vol(\Lambda_t)\right|_{t=0} = \int_{\Lambda_0}|\Delta h_0|^2 \vol.
\]
\end{pr}
\begin{proof}
Since $\Lambda_0$ is critical, using Corollary~\ref{cy:fvv} and equation~\eqref{eq:vth}, we calculate
\begin{align*}
\left.\frac{d^2}{dt^2} \Vol(\Lambda_t)\right|_{t=0} &= \left.\frac{d}{dt}\right|_{t=0} \int_{\Lambda_t} h_t\, d^*\!d\vartheta_t \vol \\
&= \int_{\Lambda_t} h_0\, d^*\!d\Delta h_0 \vol \\
& = \int_{\Lambda_t} |\Delta h_0|^2 \vol.
\end{align*}
\end{proof}

\subsection{Mirror symmetry}\label{ssec:mirror}
Kontsevich's homological mirror symmetry conjecture~\cite{Ko95} asserts that the Fukaya category of a Calabi-Yau manifold $X$ is equivalent to the derived category of coherent sheaves of its mirror Calabi-Yau manifold $X^\vee.$ In particular, Lagrangian submanifolds of $X$ should be related to holomorphic vector bundles on $X^\vee.$ Carrying the analogy further, special Lagrangian submanifolds should be related by mirror symmetry to Einstein-Hermitian metrics on holomorphic vector bundles. See \cite{Fu01,Th01,Do02} for more details. In the following, we discuss mirror analogs of the functionals $\Vol$ and $\CC$ for metrics on holomorphic vector bundles.

Let $(X,J,\omega)$ be a compact K\"ahler $n$-manifold and denote the K\"ahler metric on $X$ by $g.$ Let $E \to X$ be a rank $r$ holomorphic vector bundle and let $H$ be a Hermitian metric on $E$ with Chern connection $D$ and curvature $F.$ Denote by $D = D' + D''$ the decomposition of $D$ by type. Let $\delta = \delta' + \delta''$ be the dual of $D$ with respect to the metrics $H$ and $g.$ Let
\[
\Lambda : A^*(X) \to A^*(X)[-2]
\]
be the dual of exterior multiplication by $\omega$ and write $\hat F = \Lambda F.$ The Yang-Mills functional is given by
\[
\I(H) = \frac{1}{2} \int_X |F|^2 \omega^n.
\]
We summarize the properties of $\I$ relevant to our discussion referring the reader to~\cite[Chapter IV]{Ko87} for proofs. Our discussion also draws on~\cite{Do85}. Let
\begin{align*}
\lambda &= -2n\pi\sqrt{-1}\frac{\int_X c_1(E)\wedge \omega^{n-1}}{r\int_X \omega^n}, \\
\kappa & = 2\frac{\left(n\pi\int_X c_1(E) \wedge \omega^{n-1}\right)^2}{\left(r\int_X \omega^n\right)} + 2 \pi^2 n (n-1) \int_X (2c_2(E) - c_1(E)^2)\omega^{n-2}.
\end{align*}
Then $\I$ satisfies the topological lower bound
\begin{equation}\label{eq:bog}
\I(H) = \frac{1}{2}\int_X |\hat F - \lambda \id_E|^2 \omega^n + \kappa \geq \kappa,
\end{equation}
with equality if and only if $\hat F - \lambda \id_E = 0,$ that is, $H$ is Einstein. Observe the analogy between lower bounds~\eqref{eq:mbog} for $\Vol$ and~\eqref{eq:bog} for $\I.$ Furthermore, we have the following variational formulae analogous to Corollary~\ref{cy:fvv} and Proposition~\ref{pr:svv}. Let $H_t$ be a family of metrics on $E$ and set $v_t = H_t^{-1}\frac{d H_t}{dt}.$ Let $D_t$ be the Chern connection of $H_t$ and let $F_t$ be the curvature.
\begin{lm}
The first variation of $\I$ is given by
\[
\frac{d}{dt}\I(H_t) = \sqrt{-1}\int_X \langle D'v_t,D'\hat F_t\rangle \omega^n.
\]
Moreover, $H$ is a critical point for $\I$ if and only if $\hat F$ is parallel with respect to the Chern connection $D$ of $H.$
\end{lm}
\begin{pr}
Suppose $H = H_0$ is a critical point of $\I.$ Then
\[
\left.\frac{d^2}{dt^2} \I(H_t)\right|_{t = 0} = \int_X |\delta'D' v|^2 \omega^n.
\]
\end{pr}

We turn to the functional on the space of metrics introduced by Donaldson in~\cite{Do85,Do87}. We omit proofs referring the reader to the original papers as well as~\cite[Chapter VI]{Ko87}. Let $H = \{H_t\}_{t \in [0,1]}$ be a path in the space of Hermitian metrics on $E$ and write $v_t = H_t^{-1} \frac{dH_t}{dt}.$
\begin{pr}\label{pr:M}
The functional
\[
\MM(H) = 2\sqrt{-1}\int_0^1 \int_X \tr\left(v_t \left(F_t - \frac{\lambda}{n}\omega\id_E\right)\right)\wedge \frac{\omega^{n-1}}{(n-1)!}
\]
depends only on the endpoint preserving homotopy class of $H.$
\end{pr}
Observe the parallel between Proposition~\ref{pr:M} for $\MM$ and Theorem~\ref{tm:main} for $\CC.$ More precisely, $(F_t - \frac{\lambda}{n} \omega\id_E) \wedge \frac{\omega^{n-1}}{(n-1)!}$ is the moment map of the action of the unitary gauge group on the space of connections of $E$ with respect to the symplectic form
\[
(a,b) \mapsto \int_X \tr(a \wedge b)\wedge\frac{\omega^{n-1}}{(n-1)!}, \qquad a,b \in A^1(End(E)).
\]
The functional $\MM$ is obtained by integrating the moment map along a path of complex gauge transformations, or equivalently, a path of metrics. On the other hand, Thomas~\cite{Th01} has sketched a mirror picture in which $\im \Omega$ is the moment map and an exact path of Lagrangian submanifolds corresponds to a path of metrics. So, the definition of $\CC$ is completely analogous to the definition of $\MM.$ The author plans to discuss the relation between $\CC$ and symplectic reduction at greater length in future work.

We proceed to the variational formulae of $\MM,$ which are analogous to Corollary~\ref{cy:sl} and Proposition~\ref{pr:svc}. Since the space of Hermitian metrics on $E$ is contractible, we write $\MM(H) = \MM(H_0,H_1).$ Fix a reference metric $K$ and a family $H_s$ with velocity vector $v_s = H_s^{-1}\frac{dH_s}{ds}$ and curvature $F_s.$
\begin{lm}
We have
\[
\frac{d}{ds}\MM(K,H_s) = 2\sqrt{-1}\int_X \tr\left(v_s \left(F_s - \frac{\lambda}{n}\omega\id_E\right)\right)\wedge \frac{\omega^{n-1}}{(n-1)!}.
\]
So, a metric $H$ is a critical point of $\MM$ if and only if it is Einstein.
\end{lm}
\begin{lm}
Suppose $H=H_0$ is a critical point of $\M$. Let $D$ be the Chern connection of $H$ and write $v = v_0.$ Then
\[
\left.\frac{d^2}{ds^2}\MM(K,H_s)\right|_{s = 0} = 2\int_X |D'v|^2 \frac{\omega^n}{n!} = \int_X |Dv|^2 \frac{\omega^n}{n!}.
\]
\end{lm}
The space $\herm^+(E)$ of positive definite Hermitian forms on $E$ is analogous to the exact isotopy class $\OO^+ \subset \L^+$ of Section~\ref{ssec:gc}. It carries a Riemannian metric defined at a point $H \in \herm^+(E)$ by the formula
\[
(\xi,\zeta) = \int_X \tr(H^{-1}\xi H^{-1}\zeta) \omega^n,\qquad \xi,\zeta \in T_H \herm^+(E).
\]
We have the following parallels of Definition~\ref{df:eg}, Proposition~\ref{pr:fvE}, and Proposition~\ref{pr:con}.
\begin{df}
Let $H : [a,b] \to \herm^+$ be a path, and let $v_t = H_t^{-1}\frac{d}{dt}H_t.$ The \emph{energy} of $H$ is given by
\[
\E(H) = \int_a^b \left(\frac{dH_t}{dt},\frac{dH_t}{dt}\right)dt = \int_a^b \int_X \tr(v_t^2) \omega^n dt.
\]
We call $H$ a \emph{geodesic} if $v_t$ is constant.
\end{df}
\begin{lm}
A path $H : [a,b] \to \herm^+(E)$ is a geodesic if and only if it is a critical point of the energy functional with respect to proper variations.
\end{lm}
\begin{lm}
Fix $K \in \herm^+(E),$ and let $H : [a,b] \to \herm^+(E)$ be a geodesic with $H_s^{-1}\frac{d}{ds}H_s = v.$ Then
\[
\frac{d^2}{ds^2} \MM(K,H_s) = \int_X |D'v|^2 \frac{\omega^n}{n!} \geq 0
\]
with equality if and only if $v$ is a holomorphic endomorphism of $E.$
\end{lm}
\begin{rem}
If $L \subset X$ is a connected spin unobstructed Lagrangian submanifold, then the degree zero Floer cohomology $HF^0(L,L)$ is at most one-dimensional~\cite[Theorem~D]{FO09}. So $L$ is analogous under mirror symmetry to a holomorphic vector bundle $E \to X^\vee,$ with automorphisms only the scalar multiples of $\id_E.$ If $v$ in the preceding lemma is a multiple of $\id_E$, then $H_t$ is a real scalar multiple of $H_0.$
\end{rem}

\section{Lagrangian flux}\label{sec:flux}
Let $\Lambda : [0,1] \to \L$ be a piecewise smooth Lagrangian path and let $\ell : S^1 \to \Lambda_0$ be a loop. Let $f : [0,1] \to \X$ be a lifting of $\Lambda$ and let $\bar \ell : S^1 \times [0,1] \to X$ be the map defined by
\[
\bar\ell(u,t) = f_t(f_0^{-1}\circ \ell(u)).
\]
Define
\[
\flux(\Lambda,\ell) = \int_{S^1 \times [0,1]} \bar\ell^*\omega.
\]
\begin{lm}\label{lm:flux}
The functional $\flux(\Lambda,\ell)$ depends only on the end-point preserving homotopy class of $\Lambda$ and the class $\ell_*([S^1]) \in H_1(\Lambda_0).$
\end{lm}
\begin{proof}
$\Lambda_0$ and $\Lambda_1$ are Lagrangian, so $\omega$ gives rise to a relative cohomology class $[\omega] \in H^2(X,\Lambda_0 \cup \Lambda_1;\R).$ Moreover, writing $C = S^1 \times [0,1],$ we check below that $\bar \ell_*([C,\partial C]) \in H_2(X,\Lambda_0\cup \Lambda_1)$ depends only on the end-point preserving homotopy class of $\Lambda$ and the homology class of $\ell.$ So, the lemma follows from the identity
\[
\flux(\Lambda,\ell) = [\omega] \cap \ell_*([C,\partial C]).
\]

It remains to check that $\overline \ell_*([C,\partial C])$ depends only on the end-point preserving homotopy class of $\Lambda$ and the homology class of $\ell.$ Indeed, suppose that $f_{0,},f_{1,}: [0,1] \to \X$ are liftings of $\Lambda$ and let $\bar \ell_0,\bar \ell_1,$ be the associated maps $S^1 \times [0,1] \to X.$ Let
\[
g_t = f_{0,0} \circ f_{0,t}^{-1}\circ f_{1,t} \circ f_{1,0}^{-1} : \Lambda_0 \longrightarrow \Lambda_0.
\]
Clearly, $g_0 = \id_{\Lambda_0}.$ Let $\tilde\ell : S^1 \times [0,1]^2 \to X$ be the smooth family of maps defined by
\[
\tilde \ell(u,t,s) = f_{0,t}\circ f_{0,0}^{-1} \circ g_{st} \circ \ell(u).
\]
Since $\tilde \ell(u,t,i) = \bar \ell_i(u,t)$ and $\tilde \ell(u,i,s) \in \Lambda_i$ for $i = 0,1,$ it follows that $(\bar \ell_0)_*([C,\partial C]) = (\bar \ell_1)_*([C,\partial C]).$
The proof of independence of other choices is similar, and we leave it to the reader.
\end{proof}

In light of Lemma~\ref{lm:flux} and the isomorphism
\[
H^1(\Lambda_0,\R) \simeq Hom(H_1(\Lambda_0,\R),\R),
\]
we define
\[
\flux([\Lambda])\in H^1(\Lambda_0,\R)
\]
by $\flux([\Lambda])(\ell_*([S^1])) = \flux(\Lambda,\ell).$

\begin{rem}
Let $(M,\omega_M)$ be a symplectic manifold, let $X = M \times M$ and $\omega = -\omega_M \boxplus \omega_M.$ Let $\phi : [0,1] \to \symp(M,\omega_M)$ be a path of symplectomorphisms with $\phi(0) = \id.$ Calabi~\cite{Ca70} defined
\[
\flux([\phi]) \in H^1(M;\R).
\]
See also~\cite{Ba78} and~\cite[Chapter 10]{MS98}. Let $\Lambda$ be the path in $\L(X,M,\cdot)$ corresponding to $\phi.$ It is easy to see that $\flux([\Lambda]) = \flux([\phi]).$
\end{rem}
Let $\{\Lambda_{s,}\}_{s\in [0,1]}$ be a smooth family of Lagrangian paths. Assume $\Lambda_{s,0} = \Lambda_0$ is fixed. Let $\alpha_{s,t} \in A^1(\Lambda_{s,t})$ be the family of $1$-forms given by
\[
\alpha_{s,t} = \frac{\partial}{\partial s} \Lambda_{s,t}.
\]
Let $f_{s,}$ be a lifting of $\Lambda_{s,}$ to $\X.$
\begin{lm}\label{lm:v1f}
We have
\[
\frac{\partial}{\partial s} \flux([\Lambda_{s,}]) = [(f_{s,0})_*f^*_{s,1}\alpha_{s,1}].
\]
\end{lm}
\begin{proof}
Let $\ell : S^1 \to \Lambda_0$ and let $\bar \ell_s : S^1 \times [0,1] \to X$ be the smooth family of maps defined by
\[
\bar \ell_s(u,t) = f_{s,t}(f_{s,0}^{-1} \circ \ell(u)).
\]
Let $v_s$ be the vector field along $\bar\ell_s$ given by $v_s = \frac{d}{ds} \bar\ell_s.$
By Remark~\ref{rem:car} and Stokes' theorem, we have
\begin{align*}
\frac{\partial}{\partial s} \flux([\Lambda_{s,}])(\ell_*([S^1])) &= \frac{\partial}{\partial s} \int_{S^1 \times [0,1]} \bar \ell_s^*\omega \\
&= \int_{S^1 \times [0,1]} d i_{v_s} \omega \\
& = \int_{\partial(S^1 \times [0,1])} i_{v_s} \omega \\
& = \int_{S^1} \ell^* (f_{s,0})_*f_{s,1}^*\alpha_{s,1}.
\end{align*}
\end{proof}
The following corollary is immediate.
\begin{cy}\label{cy:ex}
The path $\Lambda_{,1} = \{\Lambda_{s,1}\}_{s \in [0,1]}$ is exact if and only if $\flux([\Lambda_{s,}])$ is constant.
\end{cy}

For simplicity, in the following we assume that $X$ and $L$ are compact. Fix $\Lambda_* \in \L.$ Let $G_* \subset H^1(\Lambda_*,\R)$ denote the subgroup given by
\[
G_* = \{\flux([\Lambda])| \Lambda_0 = \Lambda_1 = \Lambda_*\}.
\]
It seems natural to investigate necessary and sufficient conditions for $G_*$ to be discrete. The analogous question for the flux of Hamiltonian symplectomorphisms, known as the flux conjecture, was resolved unconditionally in the affirmative by Ono~\cite{On06}.
Furthermore, we have the following implication.
\begin{lm}\label{lm:cl}
If $G_*$ is discrete then the $\ham(X,\omega)$ orbit of $\Lambda_*$ is closed in $\L$ in the $C^1$ topology.
\end{lm}
Ono~\cite{On07} proved the $\ham(X,\omega)$ orbit of $\Lambda_*$ to be closed under the assumption that $\Lambda_*$ is unobstructed and has vanishing Maslov class.

Let $\L_*$ be the path connected component of $\Lambda_*$ in $\L$ and let
\[
\overline\flux : \L_* \longrightarrow H^1(\Lambda_*)/G_*
\]
be the map given by $\overline \flux(\Gamma) = [\flux([\Lambda])]$ for $\Lambda$ a path in $\L_*$ connecting $\Lambda_*$ and $\Gamma.$
\begin{proof}[Proof of Lemma~\ref{lm:cl}]
Suppose $G_*$ is discrete. Then $H^1(\Lambda_*)/G_*$ is a manifold. It is not hard to check that $0$ is a regular value of $\overline \flux.$ So $\overline\flux^{-1}(0) \subset \L_*$ is a closed submanifold. By Corollary~\ref{cy:ex}, the $\ham(X,\omega)$ orbit of $\Lambda_*$ is the path connected component of $\Lambda_*$ in $\overline\flux^{-1}(0).$ So, it must also be closed.
\end{proof}
In the following, we denote by
\[
r : H^1(X) \longrightarrow H^1(\Lambda_*)
\]
the restriction map. Assuming that $r$ is surjective, we obtain further implications of $G_*$ being discrete from the following lemmas. We consider a path $\Lambda : [a,b] \to \L$ with $\Lambda_0 = \Lambda_*$ and write $\alpha_t = \frac{d}{dt}\Lambda_t.$
\begin{lm}\label{lm:symp}
Suppose $r$ is surjective, and let $s$ be a right inverse to $r.$ There exists a path of symplectomorphisms $\phi = \{\phi_t\}_{t \in [0,1]}$ of $X$ generated by a family of vector fields $\{\xi_t\}_{t\in [0,1]}$ such that
\[
\phi_t(\Lambda_*) = \Lambda_t, \qquad s\circ (\phi_t|_{\Lambda_*})^*([\alpha_t]) = [ i_{\xi_t}\omega].
\]
In particular,
\[
\flux([\phi]) = s(\flux([\Lambda])).
\]
\end{lm}
\begin{proof}
Fix a Riemannian metric on $X.$ Let $\hat \alpha_t$ be the harmonic representative of $s\circ (\phi_t|_{\Lambda_*})^*([\alpha_t]).$ Choose a smooth family of functions $q_t : \Lambda_t \to \R$ such that
\[
dq_t = \alpha_t - \hat \alpha_t|_{\Lambda_t}.
\]
Using the fact that the graph of $\Lambda$ in $X \times [0,1]$ is a submanifold, extend $q_t$ to a family of functions $Q_t$ on $X.$ Setting $\tilde \alpha_t = \hat\alpha_t + dQ_t,$ it follows that
\[
\tilde \alpha_t|_{\Lambda_t} = \alpha_t.
\]
Let $\xi_t$ be the symplectic vector field defined by $i_{\xi_t} \omega = \tilde\alpha_t$ and let $\phi_t$ be the corresponding symplectic isotopy. Then a slight generalization of Lemma~\ref{lm:un} implies that $\phi_t(\Lambda_*) = \Lambda_t,$ and
\[
[i_{\xi_t}\omega] = [\tilde \alpha_t] = [\hat \alpha_t] = s\circ (\phi_t|_{\Lambda_*})^*([\alpha_t]).
\]
For the final claim, note that by Lemma~\ref{lm:v1f} we have
\[
\flux([\Lambda]) = \int_0^1 [(\phi_t|_{\Lambda_*})^* \alpha_t] dt.
\]
So,
\begin{align*}
\flux([\phi]) & = \int_0^1 [i_{\xi_t}\omega]dt \\
&= \int_0^1 s \circ (\phi_t|_{\Lambda_*})^*[\alpha_t] dt \\
&= s \left(\int_0^1 [(\phi_t|_{\Lambda_*})^* \alpha_t] dt\right) \\
&= s (\flux([\Lambda])).
\end{align*}
\end{proof}

\begin{lm}\label{lm:flux0}
Suppose $r$ is surjective. Then $\flux(\Lambda) = 0$ if and only if $\Lambda$ is homotopic with endpoints fixed to an exact path.
\end{lm}
\begin{proof}
The ``if'' part of the lemma follows from Corollary~\ref{cy:ex}. For the opposite implication, we proceed as follows. Choose a right inverse $s$ of $r$ and let $\phi$ be as in Lemma~\ref{lm:symp}. So, $\flux([\phi]) = s(\flux([\Lambda])) = 0.$ Thus by~\cite[Theorem 10.12]{MS98} we know that $\phi$ is homotopic with endpoints fixed to a Hamiltonian isotopy. Applying the homotopy of $\phi$ to $\Lambda_*,$ we obtain the desired homotopy of $\Lambda.$
\end{proof}

\begin{cy}\label{cy:of0}
Suppose $r$ is surjective. Then for $\Gamma \in \L_*,$ we have $\overline \flux(\Gamma) = 0$ if and only if $\Gamma$ is Hamiltonian isotopic to $\Lambda_*.$
\end{cy}
\begin{proof}
Suppose $\Gamma \in \L$ satisfies $\overline\flux(\Gamma) = 0.$ Let $\Lambda : [0,1] \to \L_*$ with $\Lambda_0 = \Lambda_*$ and $\Lambda_1 = \Gamma.$ After possibly replacing $\Lambda$ by its composition with a loop based at $\Lambda_*,$ we may assume $\flux([\Lambda]) = 0.$ So, the corollary follows from Lemma~\ref{lm:flux0}.
\end{proof}

\begin{cy}\label{cy:quot}
Suppose $r$ is surjective. Then $\overline\flux$ induces a bijection $\L_*/\ham(X,\omega) \simeq H^1(\Lambda)/G_*.$ If $G_*$ is discrete, the bijection is a homeomorphism. In particular, $\L_*/\ham(X,\omega)$ is Hausdorff.
\end{cy}
\begin{proof}
First, we show that $\overline\flux : \L_* \to H^1(\Lambda)/G_*$ is surjective.
Given $a \in H^1(X;\R),$ it is easy to construct a path of symplectomorphisms $\phi : [0,1] \to \symp(X,\omega)$ with $\phi_0 = \id$ such that $\flux([\phi]) = a.$ Define $\Lambda : [0,1] \to \L$ by $\Lambda_t = \phi_t(\Lambda_*).$ Then $\flux([\Lambda])  = r(a).$ Since $a$ was arbitrary and $r$ is surjective, we conclude that $\overline \flux$ is surjective. By Corollary~\ref{cy:of0}, the induced map is one to one. If $G_*$ is discrete, then the quotient map $H^1(\Lambda_*) \to H^1(\Lambda_*)/G_*$ admits continuous local right inverses. So the proof of surjectivity of $\overline\flux$ shows that $\overline\flux$ admits continuous local right inverses.
\end{proof}

Ono~\cite{On07} showed that if $r$ is surjective and $\Lambda_*$ is unobstructed with vanishing Maslov class, then $\L_*/\ham(X,\omega)$ is Hausdorff. See~\cite{Fu01} for an explanation of how this relates to the stability of $\Lambda_*.$

\bibliographystyle{../amsabbrvc}
\bibliography{../ref}

\vspace{.5 cm}
\noindent
Institute of Mathematics \\
Hebrew University, Givat Ram \\
Jerusalem, 91904, Israel \\

\end{document}